\documentclass[reqno,12pt,a4paper]{amsart}
\usepackage{pgfplots}
\usepackage[english]{babel}
\usepackage{amsmath,amsthm,amssymb,amsfonts}
\usepackage{scrtime, cancel}
\usepackage{enumitem, color}
\usepackage{mathtools}
\usepackage{ifpdf}
\ifpdf
\else
\fi
\usepackage{pdfsync,verbatim}
\usepackage{color}
\numberwithin{equation}{section}   

\allowdisplaybreaks
\newtheorem{theorem}{Theorem}[section]
\newtheorem{lemma}[theorem]{Lemma}
\newtheorem{proposition}[theorem]{Proposition}

\theoremstyle{definition}

\newcommand{\R}{\mathbb{R}}
\newcommand{\N}{\mathbb{N}}
\newcommand{\Z}{\mathbb{Z}}
\def\misgausskd{d \gamma^{k}_\infty}
\def\misgaussk{\gamma^{k}_\infty}
\DeclareMathOperator{\tr}{tr}             
\newcount\dotcnt\newdimen\deltay
\def\Ddot#1#2(#3,#4,#5,#6){\deltay=#6\setbox1=\hbox to0pt{\smash{\dotcnt=1
\kern#3\loop\raise\dotcnt\deltay\hbox to0pt{\hss#2}\kern#5\ifnum\dotcnt<#1
\advance\dotcnt 1\repeat}\hss}\setbox2=\vtop{\box1}\ht2=#4\box2}

\title[Normal Ornstein-Uhlenbeck semigroups
  ]{
The maximal operator of a
 normal   
Ornstein--Uhlenbeck  semigroup is of  weak type $(1,1)$ }
\author{Valentina Casarino}
\address{Universit\`a degli Studi di Padova\\Stradella san Nicola 3 \\I-36100 Vicenza \\ Italy}
\email{valentina.casarino@unipd.it}
\author{Paolo Ciatti}
\address{Universit\`a degli Studi di Padova\\Via Marzolo 9 \\I-35100 Padova \\ Italy}
\email{paolo.ciatti@unipd.it}
\author{Peter Sj\"ogren}
\address{Mathematical Sciences  University of Gothenburg 
\\ Mathematical Sciences  Chalmers University of
Technology  \\ SE - 412 96 G\"oteborg, Sweden}
\email{peters@chalmers.se}

\date{\today, \thistime}

\begin{document}
{{
\begin{abstract}
{{
Consider a normal Ornstein--Uhlenbeck
semigroup in Euclidean space, whose covariance is given by a positive definite  matrix.
The drift matrix is assumed to have eigenvalues only in the left half-plane.
We prove that the associated maximal operator
is of weak type $(1,1)$
with respect to the invariant measure.
This extends earlier work by G. Mauceri and L. Noselli.   
The proof  
goes via  the  special case where  the
matrix defining the covariance  is
$I$
and the
drift matrix
is diagonal.
}}

\end{abstract}

\maketitle

\section{Introduction}\label{intro}
Let
 $Q$ be  a real, symmetric and positive definite $n\times n$ matrix, and
$B$ a real $n\times n$ matrix whose eigenvalues have negative real parts; here $n\ge 1$.
One defines the covariance  matrices 
$$Q_t=\int_0^t
e^{sB}\,Q\,
e^{sB^*}ds\,,\qquad \text{ $t\in (0,+\infty]$, 
}$$
and the family of Gaussian measures in $\R^n$
$$
d\gamma_t (x)=
(2\pi)^{-\frac{n}{2}}
(\text{det} \, Q_t)^{-\frac{1}{2} }
e^{-\frac12 \langle Q_t^{-1}x,x\rangle}
dx  \,,\qquad \text{ $t\in (0,+\infty]$. }$$
Here   $\gamma_\infty$ is the unique invariant measure.

On the space $\mathcal C_b(\R^n)$
of bounded continuous functions,
we consider  the 
Ornstein--Uhlenbeck semigroup 
$\big(
\mathcal H_t^{Q,B}
\big)_{t> 0}\,,$ explicitly given
 by the Kolmogorov formula 
\begin{equation*}
\mathcal H_t^{Q,B}
f(x)=
\int  
f(e^{tB}x-y)d\gamma_t (y)\,, \quad x\in\R^n\,,
\end{equation*}
(see \cite{Kolmo}).
Its
  infinitesimal generator 
is given by
$$\mathcal L^{Q,B} f=
\frac12
\tr
\big(
Q\nabla^2
f\big)+
\langle Bx, \nabla f \rangle 
\,,\qquad
{\text{ $f\in
\mathcal S (\R^n)$,}}
$$
and $\mathcal  S(\R^n)$ is a core of $\mathcal L^{Q,B}$.
Here  $Q\nabla^2 f$ denotes the product 
of $Q$ and the Hessian matrix of $f$.

The relevance of
 this 
 semigroup
 is also due 
to the fact that $\big(
\mathcal H_t^{Q,B}
\big)_{t> 0}$ 
is the transition semigroup of the Ornstein-Uhlenbeck process
$$\mathcal X (t,x)=e^{tB}x+\int_0^t e^{(t-s)B}
dW(s)\,$$
on $\R^n$, where   
$W$ denotes an $n$-dimensional Brownian motion with
covariance 
matrix
$Q$. This process 
describes the random motion of a particle 
subject to friction; cf. \cite{OU} or \cite{DaPratoZ}.

Among its various properties, 
we only recall here that $\big( \mathcal H_t^{Q,B}\big)_{t>0}$ is strongly continuous in $\mathcal C_0(\R^n)$
 and in  $L^p (\R^n)$ for all $1\le p<\infty$ \cite{DaPrato, Lunardi, Lorenzi}, while
strong continuity fails 
to hold in the space of bounded, uniformly continuous functions in $\R^n$ endowed with the supremum norm (\cite[Lemma 3.2]{DaPrato}, \cite{vanNeerven}).
For some relevant results about differentiability and analiticity
 of $\big( \mathcal H_t^{Q,B}\big)_{t>0}$ in the $L^p$ spaces, we refer the reader to \cite{Lorenzi, Priola}.

We consider
the maximal operator
\begin{equation}\label{def-op-max}
\mathcal H_*^{Q,B} f(x)
=\sup_{t> 0}
\big|
\mathcal H_t^{Q,B} f(x)
\big|\,,\qquad t>0,
\end{equation}
which is 
an essential tool  in the study of the almost everywhere convergence of
 $
\mathcal H_t^{Q,B} f $
  as $t\to 0$ for $f\in L^p(\gamma_\infty$), $1\le p<\infty$.

The boundedness 
properties of $\mathcal H_*^{Q,B}$
are essentially known   when
 $\big(
\mathcal H_t^{Q,B}
\big)_{t> 0}\,$
is  {\it{symmetric}}, i.e., when $\mathcal H_t^{Q,B}$ is self-adjoint on $L^2 (\gamma_\infty)$ for all $t> 0$.
Indeed, 
for  $1<p\le \infty$, the boundedness of 
$\mathcal H_*^{Q,B}$ on $L^p (\gamma_\infty)$ then
follows from the general Littlewood--Paley--Stein theory for symmetric semigroups of contractions on Lebesgue spaces \cite{Stein}.

G. Mauceri and L. Noselli \cite{Mauceri-Noselli}
addressed
the nonsymmetric case, assuming only that $\big(
\mathcal H_t^{Q,B}
\big)_{t> 0}\,$
is {\it{normal}},
i.e., that
$\mathcal H_t^{Q,B}$ is for each $t> 0$
 a  normal operator on $L^2 (\gamma_\infty)$.
Then,
by generalizing Stein's results to a semigroup of normal contractions
whose infinitesimal generator is a sectorial operator of angle less than $\pi/2$,
they were able to prove that 
$\mathcal H_*^{Q,B}$ is bounded on $L^p (\gamma_\infty)$, for all $1<p\le \infty$.

Since the operator $\mathcal H_*^{Q,B}$ is always unbounded on $L^1 (\gamma_\infty)$,
one is led to analyze the weak type $(1,1)$ of the maximal operator. This means  seeking  an estimate
of the form
$$\gamma_\infty \{
x\in\R^n: \, 
\mathcal H_*^{Q,B}f(x)
>\alpha\}
\lesssim  \frac{\|f\|_1}{\alpha}\,,$$
holding for all $\alpha>0$ and all $f\in L^1 (\gamma_\infty)$.
In the  special case $Q=I$ and $B=-I$, 
which is symmetric,
this was proved   by
B. Muckenhoupt in the one-dimensional case  \cite{Muckenhoupt}
 and by the third author in higher dimension \cite{Peter}; 
the proof in \cite{Peter} was then simplified by
T. Men\'arguez, S. P\'erez and F. Soria
 \cite{Soria} (see also \cite{SoriaCR, Soria-Perez}).
 Another simple argument is given in \cite{JLMS}.
 For a nice discussion of the different techniques we refer the reader to \cite{Scotto}. 
 
In \cite{Mauceri-Noselli}
 Mauceri and Noselli  applied a factorization known from \cite{MPRS}, saying that  an arbitrary normal 
Ornstein--Uhlenbeck semigroup $\big(
\mathcal H_t^{Q,B}
\big)_{t> 0}\,$ can be written
as the product of more elementary semigroups, called building blocks. Each building block  is an Ornstein--Uhlenbeck semigroup
with $Q=I$ and 
$B=\lambda (R-I)$, for some  positive $\lambda$
and a real skew-adjoint matrix $R$. 
Mauceri and Noselli 
were able to prove that
for such a building block 
 the truncated maximal operator, defined by taking the supremum in \eqref{def-op-max}
 only over $0< t\le T<\infty$,
 is of weak type $(1,1)$.
If, in addition, $R$ generates a periodic group, they proved that
 the full maximal operator $ H_*^{Q,B} $  is of weak type $(1,1)$.
The
 case when the semigroup involves 
 several building blocks 
 seems not to have
 been considered as yet. Indeed,
 Mauceri and Noselli write 
 ``already the case where $B$ is a diagonal matrix with at least two different eigenvalues 
seems to require new ideas".

In this paper, we give the complete solution
of the problem studied  in  \cite{Mauceri-Noselli},
as follows.
 \begin{theorem}\label{main-theorem}
The maximal operator $\mathcal H_*^{Q,B}$
of 
 an arbitrary
 normal Ornstein--Uhlenbeck semigroup $\big(
\mathcal H_t^{Q,B}
\big)_{t> 0}$ 
is of weak type $(1,1)$ with respect to the invariant measure $\gamma_\infty$.
\end{theorem}

We first consider the special case
when $Q=I$
and 
$B=\text{diag}{\big(-\lambda_1, -\lambda_2, \ldots, -\lambda_n\big)}
$,
with $\lambda_j>0$ for  $j=1,\ldots, n$,
and
state in Theorem \ref{weaktype1} the
weak type $(1,1)$ of 
$ H_*^{Q,B} $.
{The proof of this result involves some geometry and occupies most of this paper.}
Theorem \ref{weaktype1} already extends the results in \cite{Mauceri-Noselli},
and forms the basis of the proof of Theorem 
\ref{main-theorem}.
\medskip

The paper is organized as follows.
In  Section
\ref{particular} we introduce the notation, in particular for the relevant Mehler kernel $K_t(x,u)$, 
and state the intermediate result Theorem \ref{weaktype1}.
Sections \ref{localization-section}, 
\ref{mixed-section},  \ref{mixed. t large}, and  \ref{smallt}
are devoted to the  proof of Theorem \ref{weaktype1}.
More precisely,
in Section \ref{localization-section}
we   
introduce a localization procedure
for those coordinates in which the variables $x$ and $u$
are close to each other.
In Section \ref{mixed-section}, we consider the remaining variables, and reduce the problem
to an ellipsoidal annulus.
A system of polar-like coordinates is also introduced.
Then
we prove 
in Section  \ref{mixed. t large} 
the weak type $(1,1)$
for that part of the maximal operator 
given by large
$t$.
Section \ref{smallt} is devoted to the more delicate
part corresponding to small  $t$.
Finally, in Section \ref{building}
we consider the  building blocks of an arbitrary
normal 
 Ornstein--Uhlenbeck semigroup,
  and deduce Theorem \ref{main-theorem}
from Theorem   \ref{mehler1111}, which is a slight generalization of Theorem
\ref{weaktype1}.

\medskip

In the following, 
we shall use the symbols
$c\,$ and $C$ with $0<c\,$, $C<\infty$
to denote constants 
which are not necessarily 
equal 
at 
different 
occurrences. They  depend only on the dimension and the parameters of the semigroup considered.
The symbol $\simeq$ between two positive expressions means that their
ratio is bounded above and below by such constants.
For  two positive quantities $a$ and $b$, we  write $a\lesssim b$  instead of $a\le C b$
and
$a\gtrsim b$  for $b\lesssim a$.
{The symbol $|E|$ will denote the Lebesgue measure of a measurable set $E$.}
By $\N$ we mean the set of all nonnegative integers. 
Finally, 
we write
{$\lfloor x\rfloor$
 to denote the greatest integer smaller than or equal to $x\in \R$.}
 \medskip

\noindent {\bf{Acknowledgements.}}
 The first and the second authors were partially supported by GNAMPA (Project 2016 
``Functional calculus for hypoelliptic operators on manifolds"
and 
Project 2017 ``Harmonic analysis and spectral theory of laplacians") and MIUR (PRIN 2016 ``Real and Complex Manifolds: Geometry, Topology
and Harmonic Analysis").
This research was carried out while the third author was visiting the
University of Padova, Italy. He would like to thank
the Department of Mathematics for the hospitality.
 
 \bigskip
 
\section{{Restriction to a  special case }}\label{particular}
In this and the following four sections, 
we consider
the case when
$Q=I$ and
\begin{equation}\label{Bdiag}
B=\text{diag}{\big(-\lambda_1, -\lambda_2, \ldots, -\lambda_n\big)}\,,\end{equation}
with $\lambda_j>0$ for $j=1,\ldots, n$.
We set $\lambda_{\mathrm{max}} =\mathrm{max}\, \lambda_j $
and 
$\lambda_{\mathrm{min}} =\mathrm{min} \, \lambda_j$.
 
Then the covariance matrices and the Gaussian measures are given by
$$Q_t=
\text{diag}{\left(
\frac{1}{2\lambda_1} (1-e^{-2 \lambda_1 t}),
\frac{1}{2\lambda_2} (1-e^{-2 \lambda_2 t}),
\ldots,
\frac{1}{2\lambda_n} (1-e^{-2 \lambda_n t})\right)
}$$
and
\begin{align*}
d\gamma_t (x)&
=\pi^{-\frac{n}{2}}
\frac{
\sqrt{\Pi_{j=1}^n
\lambda_j}}{\sqrt{\Pi_{j=1}^n(1-e^{-2 \lambda_j t})}
}
\exp \Big(
{-\sum_{j=1}^{n}
\frac{\lambda_j}{1-e^{-2 \lambda_j t}}\,x_j^2 
}\Big)
dx_1\ldots dx_n.\end{align*}
The invariant measure is
\begin{align}\label{measure}
d\gamma_\infty(x) &=\pi^{-\frac{n}{2}}
\sqrt{\Pi_{j=1}^n
\lambda_j}\,
\exp \Big(
{-\sum_{j=1}^{n}
{\lambda_j}x_j^2   }\Big)
dx_1\ldots dx_n.\end{align}

We denote the Ornstein--Uhlenbeck semigroup
simply by
 $\mathcal H_t$, suppressing the indices $Q,B$. 
It
may be written as
\begin{align*}
&\mathcal H_t
f(x)
=\pi^{-\frac{n}{2}}
\frac{
\sqrt{\Pi_{j=1}^n
\lambda_j}}{\sqrt{\Pi_{j=1}^n(1-e^{-2 \lambda_j t})}
}
\int
f(
e^{-t\lambda_1}x_1-y_1,
\ldots,
e^{-t\lambda_n}x_n-y_n
)
\\
&\qquad\qquad\qquad\qquad \times
\exp \Big(
-\sum_{j=1}^{n}
\frac{\lambda_j}{
1-e^{-2 \lambda_j t}
}
y_j^2  \Big)
\, dy_1\ldots dy_n.
\end{align*}
A straightforward computation
leads to
\begin{align*}
\mathcal H_t
f(x)
&=\frac{
\exp \Big({\sum_{j=1}^{n}
{\lambda_j}x_j^2   }\Big)
}{\sqrt{\Pi_{j=1}^n(1-e^{-2 \lambda_j t})}
}
\\
&\;\;\times
\int
f(
u_1,
\ldots,
u_n
)
\exp \Big(-{\sum_{j=1}^{n}\frac{\lambda_j}{1-e^{-2 \lambda_j t}}
(x_j-e^{-\lambda_j t}u_j)^2   }\Big)
 \, d\gamma_\infty(u_1,\ldots, u_n)
.
\end{align*}
We write this as
 \begin{align*}
 \mathcal H_t
f(x) &=
 \int
K_t
(x,u)\,
f(u)\,
 d\gamma_\infty(u)  
  \,,
\end{align*}
where  $K_t$ denotes  the Mehler  kernel,
given by 
\begin{align*}
K_t
(x,u)
&=\frac{
\exp \Big({\sum_{j=1}^{n}
{\lambda_j}x_j^2   }\Big)
}{\sqrt{\Pi_{j=1}^n(1-e^{-2 \lambda_j t})}
}
\exp
 \Big(-{\sum_{j=1}^{n}\frac{\lambda_j}{1-e^{-2 \lambda_j t}}
(x_j-e^{-\lambda_j t}u_j)^2 }\Big)
\,
\end{align*}
for $x,u\in\R^n$.
It is clearly the tensor product of
the one-dimensional kernels
\begin{align}\label{KtjQB}
K_{t,j}
(x_j,u_j)
&=
\frac{
\exp ({\lambda_j}x_j^2)  
}{\sqrt{1-e^{-2 \lambda_j t}}
}
\exp \Big(-{\frac{\lambda_j}{1-e^{-2 \lambda_j t}}
(x_j-e^{-\lambda_j t}u_j)^2 }\Big)
.
\end{align}
The maximal operator is
 \begin{align*}
\mathcal H_*
f(x)=
\sup_{t> 0}
\big|
\mathcal H_t
f(x)
\big|.
\end{align*}
 We will prove the following special case of Theorem \ref{main-theorem}.
\begin{theorem}\label{weaktype1}
If $Q=I$ and $B$ is diagonal and given by \eqref{Bdiag}, 
then 
$\mathcal H_*=\mathcal H_*^{I,B}
$
 is of weak type $(1,1)$ with respect to the invariant measure $\gamma_\infty$.
\end{theorem}

In the proof of this theorem,
 we distinguish between global and local variables.
For $k\in\{0,\ldots,n\}$
we 
define 
 \begin{align*}
M_k&=\{
(x,u)\in\R^n\times\R^n\,:  \,|x_j-u_j|> \frac{1}{1+|x_j|},\text{ $j=0,\ldots,k$}\,,\,\\
& \qquad\qquad\qquad\qquad\text{ and }\quad |x_j-u_j|\le \frac{1}{1+|x_j|},\text{ $j=k+1,\ldots,n$ }
\}.
\end{align*}
If $k=0$ or $k=n$, this means that the second or the first inequality, respectively, applies to all $j$.
We 
call the inequalities
$ |x_j-u_j|>
 \frac{1}{1+|x_j|}$ 
and
 $|x_j-u_j|\le \frac{1}{1+|x_j|}$
the global and the local condition, respectively. 
If  $(x,u)\in M_k$ for some $k\in\{0,\ldots,n\}$, we write
\begin{equation*}
x=(\xi,x_{\text{loc}}), \quad
\text{ with }
\quad
\xi=(x_1,\ldots,x_k)
\quad
\text{ and }
\quad
x_{\text{loc}}=(x_{k+1},\ldots,x_n).
\end{equation*}
Thus $x=x_{\mathrm{loc}}$
for $k=0$
and $x=\xi$ 
for $k=n$.
We use similar notation for $u$ and write 
\begin{equation*}
u=(\eta, u_{\text{loc}}), 
\quad
\text{ with }
\quad
\eta=(u_1,\ldots,u_k)
\quad
\text{ and }
\quad
u_{\text{loc}}=(u_{k+1},\ldots,u_n).
\end{equation*}

Then
let
\begin{align*}
\mathcal H_*^{k}
f(x)=\sup_{t> 0}
\big|
 \int
K_t
(x,u)\,\chi_{{M_k}}(x,u)\,f(u)
\,
 d\gamma_\infty(u) 
 \,
\big|\,,
\end{align*}
where $k\in\{0,\ldots,n\}$.

Observe that 
$\mathcal H_*^{0}$ is the local part of
$\mathcal H_*$.
 To prove Theorem \ref{weaktype1}, it is for obvious symmetry reasons 
 enough to show that each
 $\mathcal H_*^{k}$, \,\,$k=0, \ldots,n$, is
 of weak type $(1,1)$ with respect to $\gamma_\infty$.
 The proof is quite long and will be divided in several steps.

\bigskip

\section{The localization procedure}\label{localization-section}

We start by proving a simple estimate for the local coordinates.
\begin{lemma}\label{stima1glob}
If for some 
 $j\in\{1,\ldots,n\}$
the point
$(x_j,u_j) \in\R\times\R
$ 
satisfies the local condition 
$ |x_j-u_j|\le {1}/{(1+|x_j|)}$,
then
\begin{equation*}
\big|K_{t,j}
(x_j,u_j)\big|\lesssim
\,\frac{
\exp \Big({
\lambda_j x_j^2   }\Big)
}{ \big(\min (1, t)\big)^{1/2}
}
\exp{
 \Big(-c\,
 \frac{(
 x_j-u_j)^2}{\min (1, t)} \,
\Big)
}
\,,\qquad
\text{$t>0$.}
\end{equation*}

\end{lemma}
\begin{proof} 
The following argument is  well known,
see e.g.  \cite[proof of Lemma 5.3]{Mauceri-Noselli}.
We have
\begin{align}\label{local-bound}
&\frac{(x_j-e^{-\lambda_j t}u_j)^2}{1-e^{-2 \lambda_j t}}=
\frac{(x_j-u_j+u_j-e^{-\lambda_j t}u_j)^2}{1-e^{-2 \lambda_j t}}\notag\\
&\ge
\frac{(x_j-u_j)^2 -2|u_j| \,|x_j-u_j| (1-e^{-\lambda_j t})  }{1-e^{-2 \lambda_j t}}\notag\\
&\ge
\frac{(x_j-u_j)^2}{1-e^{-2 \lambda_j t}}
 -\frac{2|x_j| \,|x_j-u_j|   }{1+e^{- \lambda_j t}}
 -\frac{2(u_j-x_j)^2  }{1+e^{- \lambda_j t}}
\notag\\
&\ge
\frac{(x_j-u_j)^2}{1-e^{-2 \lambda_j t}}
 -
  \frac{2|x_j|}{1+|x_j|}
  -\frac{2  }{(1+|x_j|)^2}
\notag\\
&\ge
\frac{(x_j-u_j)^2}{1-e^{-2 \lambda_j t}}
 -
4 .
\end{align}
Inserting this in \eqref{KtjQB}, one obtains the desired conclusion.
\end{proof}
Next, we simplify
the problem by means of a localization process
for
 the local variables, covering $\R^{n-k}$
with suitable rectangles. Assume $0\le k<n$.
First we  split the real line 
into pairwise disjoint intervals of the type
\begin{equation*}
I_s =   \left(s-\frac1{1+|s|}, \: s+\frac1{1+|s|}\right].
\end{equation*}
Clearly, this can be done with values of $s$ in an increasing  sequence
$\big(s^{(\nu)}\big)_{\nu\in\Z}$. We claim that for each $s$
\begin{equation}\label{impl}
  s' \in I_s, \quad |s''-s'| \le \frac1{1+|s'|} \qquad \Rightarrow  \qquad
s'' \in 3I_s,
\end{equation}
where $3I_s$ denotes the concentric scaling of  $I_s$ by a factor 3.
Indeed, since
$ |s'-s| \le 1/{(1+|s|})$,
\begin{equation*}
 1+|s| \le 1+|s'| +\frac1{1+|s|} \le  2(1+|s'|)\,,
\end{equation*}
and it follows that
\begin{equation*}
 |s''-s| \le  |s''-s'| +  |s'-s| \le  \frac1{1+|s'|}+  \frac1{1+|s|} \le
 \frac3{1+|s|}.
\end{equation*}
Observe also that the scaled intervals $ 3I_{s^{(\nu)}},\: \nu \in \Z$,
have bounded overlap. 
A similar splitting was used in \cite{FGS}.

 Next, we apply this in each variable in $\R^{n-k}$, assuming $k<n$.
Denoting by $\nu= (\nu_{k+1},...,\nu_n) \in \Z^{n-k}$ a multiindex, we split 
 $\R^{n-k}$ into closed rectangles
\begin{equation*}
 \mathcal C_\nu
   =\prod_{j=k+1}^n
 \left[s^{(\nu_j)}-\frac1{1+|s^{(\nu_j)}|},\: 
s^{(\nu_j)}+\frac1{1+|s^{(\nu_j)}|}\right], \qquad \nu \in  \Z^{n-k},
\end{equation*}
with centers $s^\nu= (s^{(\nu_j)})_{j=k+1}^n$.
A consequence of \eqref{impl} is that
\begin{equation*}
  (x, u) \in M_k, \quad x_{\text{loc}}\in  \mathcal C_\nu\qquad 
\Rightarrow  \qquad u_{\text{loc}} \in \tilde{\mathcal C}_\nu,
\end{equation*}
where $\tilde{\mathcal C}_\nu =3\mathcal C_\nu$ is the concentric scaling. 
This implication
 assures  that the values of $\mathcal H_*^{k}f$
in $\R^k\times \mathcal C_\nu$
 only depend on the restriction of $f$ to $\R^k\times
 \tilde{\mathcal C_\nu}$.
Further, the
rectangles ${\mathcal C}_\nu$ 
are pairwise disjoint except for boundaries, and the
$\tilde{\mathcal C}_\nu$ have bounded overlap.

In each set
$\R^k\times
 \tilde{\mathcal C_\nu}$
 the Gaussian density varies  little
 with the 
local coordinates, in the following way.

\begin{lemma}\label{bounded-overlapping}
Let $\nu\in\Z^{n-k}$, $k\in\{0,\ldots, n-1\}$.
Then for any $u_{\text{loc}}\in \tilde {\mathcal C_\nu}$,
\begin{align*}
\exp\left(\sum_{j=k+1}^n\lambda_j u_j^2\right) \sim \exp(D_\nu),
\end{align*}
where
 $D_\nu=\sum_{j=k+1}^n\lambda_j (s^{(\nu_j)})^2 $.
\end{lemma}
\begin{proof}
This is a well-known and simple fact (see, for example, \cite[p.\,74]{Peter}). 
\end{proof}
To prove Theorem \ref{weaktype1}, it suffices to show for each $k\in\{0,1,\ldots,n\}$
and each $\nu\in \Z^{n-k}$
that $\mathcal H^k_*$
maps $L^1 (	\R^k\times \tilde{\mathcal C}_\nu;\, d\gamma_\infty)$ boundedly
into
$L^{1,\infty} (\R^k\times {\mathcal C}_\nu; \,d\gamma_\infty)$,
uniformly in $\nu$.
Indeed, the bounded overlap of the
$\tilde{\mathcal C}_\nu$
will then allow summing in $\nu$.
In
 the case $k=n$, there is no need for the
${\mathcal C}_\nu$
and
$\tilde{\mathcal C}_\nu$.

With $\nu$ fixed, Lemma
\ref{bounded-overlapping}
then makes it natural
to replace 
$d\gamma_\infty$ 
by the measure
\begin{align*}
\misgausskd
(x) &=\pi^{-\frac{k}{2}}
\sqrt{\Pi_{j=1}^k
\lambda_j}\,
\exp \Big(
{-\sum_{j=1}^{k}
{\lambda_j}x_j^2   }\Big)
\,dx_1\ldots dx_k\, dx_{\text{loc}},
\end{align*}
where
$dx_{\text{loc}}=dx_{k+1}\ldots dx_{n}$. Observe that
$d \gamma^{n}_\infty=
d \gamma_\infty$.

We are now led
to the kernel 
\begin{align}\label{Mehler-Di}
K_t^{k,\nu}(x,u)
&=\frac{
\exp \Big({\sum_{j=1}^{k}
{\lambda_j}x_j^2   }\Big)
}{
\sqrt{\Pi_{j=1}^n(1-e^{-2 \lambda_j t})}
}\notag\\
&\times
\exp
 \Big(-{\sum_{j=1}^{n}\frac{\lambda_j}{1-e^{-2 \lambda_j t}}
(x_j-e^{-\lambda_j t}u_j)^2 }\Big)\,\chi_{M_k
 }(x,u)\,
 \chi_{
 \mathcal C_\nu
}(x_{\mathrm{loc}})
,
\end{align}
which vanishes for $u_{\text{loc}}\not\in \tilde {\mathcal C_\nu}$,
and to the operator
\begin{equation} \label{maxfcn}
\mathcal H_*^{k,\nu}
f(x)=\sup_{t> 0}
\big|
 \int
K_t^{k,\nu}
(x,u)\,f(u)\,
 \misgausskd(u) 
 \,
\big|.
\end{equation}

As easily verified by means of a small computation, 
Theorem \ref{weaktype1} can be rephrased 
as follows.
\begin{theorem}\label{th:revisited} Let $k\in\{0,\ldots,n\}$.
For all functions  $f\in L^1( \misgaussk)$
\begin{equation} \label{thesis-mixed-Di}  
 \misgaussk
 \{x : \mathcal H_*^{k,\nu}f(x) > \alpha\} \lesssim \frac1\alpha\,
\|f\|_{L^1( \misgaussk)},\qquad \text{ $\alpha>0$,}
\end{equation}
uniformly in  $\nu\in \Z^{n-k}$. 
\end{theorem}

We first show that Theorem \ref{th:revisited} holds in 
the (entirely local) case $k=0$.
\begin{proposition}\label{propo-locale}
The maximal operator $\mathcal H_*^{0,\nu}
$ is of weak type $(1,1)$, uniformly in $\nu$.
\end{proposition}
\begin{proof} 
Lemma \ref{stima1glob}
implies that for $(x,u)\in M_0$, 
$x\in 
 {\mathcal C}_\nu$ and $u\in \tilde {\mathcal C}_\nu
$
\begin{align*}
\big|
K_t^{0,\nu}
(x,u)
\big|
&\lesssim
\frac{
1}{ \big(\min (1, t)\big)^{n/2}
}
\exp{
 \Big(-c\,
 \frac{
|x-u|^2}{\min (1, t)} \,
\Big)
}
\,,\,\,\, \qquad t>0.
\end{align*}
Standard methods now allow us 
 to estimate
$\mathcal H_*^{\text{0},\nu}f$ 
in 
$L^{1,\infty} (\mathcal C_\nu)$ in terms of the norm of $f$ in
$L^1 (
\tilde{ \mathcal C}_\nu)
$.
For further details, see
for example
 \cite[Section 3]{JLMS}.
\end{proof}
When proving Theorem \ref{th:revisited}
for $k>0$,  we can
 assume that 
 $f$ is nonnegative,  supported in $\R^k\times\tilde {\mathcal C}_\nu$
 and normalized in the sense  that
 \begin{equation*}
 \|f\|_{L^1( \misgaussk)}=1.
 \end{equation*}
The level set in \eqref{thesis-mixed-Di}
 is contained in
 $\R^k\times  \mathcal C_\nu$, and 
$ \gamma_\infty^k(\R^k\times \mathcal C_\nu) \lesssim 1$.
We may  assume  that $\alpha $ is large, since \eqref{thesis-mixed-Di}  is trivial in  the opposite case.
The meaning of ``large" here will be specified later and will depend only on the dimension and the parameters of the semigroup.

\bigskip

\section{{Some elliptic geometry}}\label{mixed-section}

\subsection{
Reduction to an ellipsoidal annulus}\label{restriction}
We  simplify the proof of  Theorem~\ref{th:revisited} 
by restricting the global variables to an  ellipsoidal annulus, defined in terms of 
the quadratic form 
\begin{equation}\label{quadratic-form}
R(\xi) = \sum_{j=1}^k \lambda_j x_j^2,
\end{equation}
where  $\xi=(x_1,\ldots,x_k)$. 
Fixing
 a large $\alpha$,
 we shall see  that 
it is not restrictive to assume that $x=(\xi, x_{\mathrm{loc}})$
in \eqref{thesis-mixed-Di} 
is such that 
 $\xi $
is in 
the set
\begin{equation}\label{crown}
{\mathcal E}=\{
\xi \in\R^k:\, \frac12  \log \alpha\le 
R(\xi)
\le 2  \log \alpha\,
\}.
\end{equation}
We first consider the set of points not verifying the inequality
$R(\xi)\le
 2\log \alpha
$, which satisfies
\begin{align}\label{restrizione-prima}
\misgaussk
\{
(\xi,x_{\text{loc}})\in
\R^k\times \mathcal C_\nu:
\,R( \xi)>
2 \log \alpha
\}
&\lesssim
\big|
\mathcal C_\nu\big|
\int_{ R(\xi)>
 2 \log \alpha }
\exp
(-R(
\xi)  )
d\xi\, 
\notag\\
&
\lesssim
{ (2 \log \alpha)^{(k-2)/2}}\,\exp (
{-
 2 \log \alpha })
\notag\\
&
\lesssim \frac1\alpha;
 \end{align}
 to get the second inequality here, one uses polar coordinates after the change of variables $x_j'=x_j \sqrt{\lambda_j}$.

 Further, we claim that for any  $(x,u)\in M_k$, 
\begin{equation}\label{claim}
 R(\xi) < \frac12 \log \alpha\; \qquad \Rightarrow \;\qquad 
K_t^{k,\nu}(x,u) \lesssim \alpha.
\end{equation}
This requires a lemma which will also be useful later;
recall that
$x = (\xi, x_{\mathrm{loc}})$.
\begin{lemma}\label{lemma-sost-peter}
  If $(x,u)\in M_k$ and $0<t\le 1$, then
\begin{equation*}
\frac1{(1+|\xi|)^2} \lesssim t^2 |\xi|^2  +
\sum_1^k (x_j-e^{-\lambda_jt} u_j)^2.
\end{equation*}
\end{lemma}
\begin{proof}
  From the definition of $M_k$ we have
\begin{align*}
\frac1{1+|\xi|}& \le \sum_1^k |x_j- u_j|
= \sum_1^k |(1-e^{\lambda_jt})x_j + e^{\lambda_jt} x_j - u_j|\\
&\lesssim t  \sum_1^k
|x_j| + \sum_1^k e^{\lambda_jt}|x_j- e^{-\lambda_jt}u_j|
\lesssim t|\xi| + \sum_1^k |x_j-e^{-\lambda_jt} u_j|.
\end{align*}
The lemma follows.
\end{proof}
To verify
 \eqref{claim}, we first assume that $t> 1$. Then because of \eqref{Mehler-Di}
\begin{equation*}
K_t^{k,\nu}(x,u) \lesssim  \, e^{R(\xi)} <  \sqrt{\alpha}
\le  \alpha,
\end{equation*}
since $\alpha$ is large.
In the case when  $t\le 1$,  we have   
  \begin{equation*}
K_t^{k,\nu}(x,u) 
 \lesssim   \frac{e^{R(\xi)}}{t^{n/2}} 
\exp
 \Big(-c\sum_{j=1}^{k}\frac{(x_j-e^{-\lambda_j t}u_j)^2}t
 \Big).
\end{equation*}
It follows from Lemma \ref{lemma-sost-peter} that 
 \begin{equation*}
 t^2 \gtrsim \frac1{(1+|\xi|)^4} \qquad  \text{or}  \qquad 
\sum_{j=1}^{k}\frac{(x_j-e^{-\lambda_j t}u_j)^2}t 
\gtrsim \frac1{(1+|\xi|)^2t}.
\end{equation*}
The first inequality here implies that
\begin{equation*}
K_t^{k,\nu}(x,u)
 \lesssim    e^{R(\xi)}\, (1+|\xi|)^n
 \lesssim   e^{2R(\xi)} <   \alpha.
\end{equation*}
If the second inequality holds, we have
\begin{equation*}
K_t^{k,\nu}(x,u) \lesssim  \frac{e^{R(\xi)}}{t^{n/2}} 
\exp
 \Big(-\frac c{(1+|\xi|)^2t}
 \Big)
 \lesssim    e^{R(\xi)}\, (1+|\xi|)^n,
\end{equation*}
and  the same estimate follows. Thus 
\eqref{claim} is verified.

Replacing $\alpha$ by $C\alpha$ for some $C$, we see from  \eqref{restrizione-prima} 
and
\eqref{claim}
that
we can assume
 $\xi
\in \mathcal E$
 in the 
proof of  Theorem \ref{th:revisited}.

\subsection{Polar-like coordinates in $\R^k$.}\label{section:polar}
Fix $\beta>0$
  and consider the ellipsoid
\begin{equation*}
E_\beta
=\{\xi\in\R^k: R(\xi)=\beta
\}.\end{equation*}
We introduce the anisotropic dilations
$$e^{\lambda s}\,\xi=
 (e^{\lambda_j\, s}\,x_j)_{ j=1}^k.$$ Then each $\xi\in\R^k\setminus\{ 0\}$ may be written in a unique way as
$\xi =e^{\lambda s}\, \tilde \xi $ with $s \in \R$ and
$\tilde \xi 
=(\tilde \xi_j )_{j=1}^k \in E_\beta$.
Thus $x=(\xi, x_{\mathrm{loc}})\in\R^n$ 
is given by
\begin{equation}\label{def-coord}
x=(e^{\lambda s}\tilde \xi,x_{\mathrm {loc}}) 
\,.
\end{equation}
The Lebesgue measure $d\xi$ in $\R^k$ satisfies
\begin{equation}\label{equiv:lebesgue}
d\xi\simeq |e^{\lambda s}\tilde \xi|
\,ds\, dS(\tilde\xi),
\end{equation}
where
$dS$ is the area measure of the ellipsoid $E_\beta$.
Indeed, we will see 
in the next subsection
that  the curve $s\mapsto e^{\lambda s}\tilde \xi$ is transverse to the family of ellipsoids defined by $R(\xi)$.

In the following result, we estimate the distance between two points in terms of the coordinates $s$, $\tilde \xi$.
\begin{lemma}\label{lemma-peter-coord}
 Let  $\xi^{(0)},\; \xi^{(1)}\in \R^k\setminus \{ 0\} $ 
and assume $ R(\xi^{(0)}) > \beta/2$.
Write  $\xi^{(0)} = e^{\lambda s^{(0)}}\, \tilde \xi^{(0)}$
and  $\xi^{(1)} = e^{\lambda s^{(1)}}\, \tilde \xi^{(1)}$ with  $s^{(0)}$, $s^{(1)}\in \R$ and
$\tilde \xi^{(0)},\; \tilde \xi^{(1)} \in E_\beta$. 
\\
(a)
Then
\begin{equation}
  \label{lem1}
  |\xi^{(0)} - \xi^{(1)}| \ge c\,  |\tilde \xi^{(0)} - \tilde \xi^{(1)}|.
\end{equation}
(b)
If also $s^{(1)} \ge 0$, then 
\begin{equation}
  \label{lem2}
  |\xi^{(0)} - \xi^{(1)}| \ge c\,\sqrt\beta\,|s^{(0)} -s^{(1)}|.
\end{equation}
\end{lemma}

\begin{proof} 
Let $\Gamma: [0,1] \to \R^k$ be a
  differentiable curve with  $\Gamma(0) =\xi^{(0)}$ and  $\Gamma(1) =\xi^{(1)}$.
It is clearly enough to bound  the length of any such curve from
below by the right-hand sides of \eqref{lem1} and \eqref{lem2}.

For each $\tau \in [0,1]$, we write  
$\Gamma(\tau) =e^{\lambda s(\tau)}\,\tilde \xi(\tau)$ with $\tilde \xi(\tau) =(\tilde \xi_j (\tau))_1^k \in E_\beta$, so that $s(0)=s^{(0)}$ and  $s(1)=s^{(1)}$. 
 The tangent vector is
\begin{equation*}
  \Gamma'(\tau) = \left( s'(\tau)\, \lambda_j\, e^{\lambda_j\, s(\tau)}\, \tilde \xi_j(\tau) 
+ e^{\lambda_j \,s(\tau)}\, \tilde \xi_j'(\tau) \right)_{j=1}^k\,,
\end{equation*}
and
\begin{eqnarray*}
  |\Gamma'(\tau)|^2 &=&\sum_1^k e^{2\lambda_j \,s(\tau)}\, 
\left(s'(\tau)\,\lambda_j \, \tilde \xi_j(\tau) 
+  \tilde \xi_j'(\tau) \right)^2  \\ 
&\ge&
\min_j e^{2\lambda_j\, s(\tau)}\, |s'(\tau)\,\lambda \, \tilde \xi(\tau)
+ \tilde \xi'(\tau)|^2,
\end{eqnarray*}
where $\lambda \, \tilde \xi(\tau)$ denotes the vector 
$(\lambda_j \, \tilde \xi_j(\tau))_{j=1}^k$. This vector is normal to $E_\beta$ at
$\tilde \xi(\tau)$ and so orthogonal to the tangent vector $\tilde \xi'(\tau)$, and
we conclude that
\begin{align}\label{2}
 |\Gamma'(\tau)|^2 \ge \min_j e^{2\lambda_j\, s(\tau)}\,
\left(s'(\tau)^2  \,|\lambda \, \tilde \xi(\tau)|^2 + |\tilde \xi'(\tau)|^2  \right).
\end{align}

We need a lower estimate of $ s(0)$. If  $s{(0)} < 0$,  the assumption 
 $ R(\xi^{(0)}) > \beta/2$ implies that
 \begin{equation*}
   \beta/2 < \sum_j \lambda_je^{2\lambda_j\, s{(0)}}(\tilde \xi_j^{(0)})^2  \le
e^{2\lambda_{\min}\,\, s{(0)}} \, R(\tilde \xi^{(0)}) = e^{2\lambda_{\min}\, s{(0)}} \, \beta\,.
 \end{equation*}
Thus we always have 
 \begin{equation*}
   s{(0)} >  -\tilde s,
 \end{equation*}
 where $\tilde s=\log 2/{(2\lambda_{\min})} $.
 
Assume now that $ s(\tau) > -2\tilde s$ for all $\tau \in [0,1]$.
Then the minimum in \eqref{2} stays away from $0$ and we get 
\begin{equation*}
 |\Gamma'(\tau)| \gtrsim |s'(\tau)|  \,|\lambda \, \tilde \xi(\tau)|
 \gtrsim \sqrt\beta\, |s'(\tau)|
 \end{equation*}
and 
\begin{equation*}
 |\Gamma'(\tau)| \gtrsim  |\tilde \xi'(\tau)|.
 \end{equation*}
Integrating each of these two estimates with respect to $\tau$ in $[0,1]$, we see that the length of $\Gamma$
is bounded below by the right-hand sides of  \eqref{lem2} and \eqref{lem1}. 

If instead   $ s(\tau) \le -2\tilde s$ for some  $\tau \in [0,1]$,
 the image $s([0,1])$ contains the interval 
$[-2\tilde s, \max (s(0), s(1))]$.
Then we can find a
closed subinterval $I \subset [0,1] $ such that for  $\tau \in I$
\begin{equation*}
 -2\tilde s \le s(\tau) \le \max (s(0), s(1))
 \end{equation*}
and, moreover, equality holds in the left-hand inequality here at one endpoint
of $I$ and in the right-hand  inequality at the other endpoint.
For the length of $\Gamma$, we now have, in view of  \eqref{2},
\begin{equation*}
 \int_0^1 |\Gamma'(\tau)|\,d\tau \ge \int_I |\Gamma'(\tau)|\,d\tau \gtrsim 
 \sqrt\beta\, \int_I |s'(\tau)|\,d\tau \ge  \sqrt\beta\: 
\big(\max\, (s(0), s(1))\; + \;2\tilde s\big).
\end{equation*}
Since $s(0) > -\tilde s$, the last quantity here is larger than $\sqrt\beta\,|\tilde s| \gtrsim \sqrt\beta \sim \mathrm{diam}\: E_\beta$. Thus the length of the 
curve is bounded below by the right-hand side of \eqref{lem1}.
If we also assume  $s^{(1)} \ge 0$, the same is true with  \eqref{lem1}
replaced by \eqref{lem2}, since then
\begin{equation*}
 \max (s(0), s(1)) +2\tilde s \ge  |s(0)- s(1)|.
\end{equation*}
The proof of the lemma 
is complete.\end{proof}

\subsection{
The Gaussian measure of a tube}\label{restriction}
We will need a geometric, $k$-dimensional lemma.
In $\R^k$ we write  points as $\xi=
(x_j)_{j=1}^k$ and use the measure
\begin{equation*}
d\mu_R(\xi) = e^{-R(\xi)}\,d\xi\,,
\end{equation*}
where $R(\xi)$ was defined in \eqref{quadratic-form}.
Recall that  $e^{\lambda t}\,\xi = (e^{\lambda_j t}\,x_j)_{j=1}^k$
and that $\alpha>0$ is large.

We fix $\beta$ with $\frac12 \log \alpha \le \beta \le 2 \log \alpha$
and 
consider a spherical cap of the ellipsoid $E_\beta$, centered at some point  $\xi^{(1)}\in E_\beta$.
Explicitly, we define
\begin{equation*}
  \Omega = \{\xi \in \R^k: R(\xi) = \beta, \; |\xi - \xi^{(1)}| < a \}
\end{equation*}
with $a>0$.
Observe that $|\xi| \simeq\sqrt \beta $ for  $\xi \in \Omega$.
Then we define the tube 
\begin{equation}   \label{zona}
Z =  \{e^{\lambda s}\,\xi: s>0,\;\xi \in \Omega \}.
\end{equation}

\begin{lemma}\label{lemma-Peter-forbidden}
The $\mu_R$
measure   of $Z$ satisfies 
\begin{equation*}
\mu_R(Z)\lesssim
\frac{a^{k-1}}{\sqrt{ \beta}}\, e^{-\beta}.
\end{equation*}
\end{lemma}
\begin{proof}
For   $s\ge 0$ the set
\begin{equation*}  
\Omega_s =\{e^{\lambda s}\,\xi: \xi \in \Omega \}
\end{equation*}
  is a slice of $Z$. The selfadjoint linear map 
\begin{equation*}    
F_s\,:\,\xi \mapsto e^{\lambda s}\xi 
\end{equation*}
 is a bijection
between  $\Omega$  and $\Omega_s$. 
To estimate $\mu_R(Z)$, we need an estimate of 
 the area $\big| \Omega_s\big|$ of the $(k-1)$-dimensional surface 
$\Omega_s$.

 A normal vector to  $\Omega_0 =\Omega $ at the point $\xi \in \Omega $ is
$v  = (\lambda_j x_j)_{j=1}^k$, and the tangent hyperplane at   $\xi$ 
is $v^\perp$. For   $s > 0$ the  tangent hyperplane of $\Omega_s$
at the point  $F_s(\xi)$ is  $F_s(v^\perp)$. Thus a normal to $\Omega_s$
at the same point is  
$w = F_s^{-1}(v) =  (e^{-\lambda_j s}\lambda_j x_j)_{j=1}^k$.
The angle $\psi(s,\xi)$ between $w$ and  
$F_s(v) =  (e^{\lambda_j s}\lambda_j x_j)_{j=1}^k$ is given by 
\begin{equation*}
  \cos \psi(s,\xi) = \frac{w\cdot F_s(v) }{\|w \|\, \|F_s(v) \|}
=  \frac{\sum_1^k \lambda_j^2 x_j^2}
{\sqrt{\sum_1^k e^{-2\lambda_j s}\lambda_j^2 x_j^2 }\: 
\sqrt{\sum_1^k  e^{2\lambda_j s}\lambda_j^2 x_j^2 }}.
\end{equation*}
We remark that 
this  shows that 
$ \cos \psi(s,\xi) $ stays away from zero; this yields the transversality
 mentioned in the preceding subsection.

Since  $F_s(v) = \partial F_s(\xi)/\partial s$,
the distance from a point  $ F_s(\xi)\in \Omega_s $ to $ \Omega_{s+h}$ in the normal
direction is, for small $h>0$, essentially
\begin{equation*}
h | F_s(v)| \cos\psi(s,\xi).
\end{equation*}
Thus the Lebesgue measure in $Z$ is given by  $| F_s(v)| \cos\psi(s,\xi)\,dS_s\,ds$,
where $dS_s$ denotes the $(k-1)$-dimensional area measure of $ \Omega_s$.
It follows that 
\begin{equation}\label{muz}
\mu_R(Z)  = \int_0^\infty \int_{\Omega_s}
| F_s(v)| \cos\psi(s,\xi)\, e^{-R(e^{\lambda s}\xi)}\,dS_s\,ds.
\end{equation}

To evaluate this, we must first estimate the area $|\Omega_s|$.
 The area of  $\Omega$ can be approximated by that of  a union of 
small $(k-1)$-dimensional simplices, i.e. small convex $k$-gons, tangent to  $\Omega$. Similarly,
that of  $\Omega_s$ is approximated by 
the images under $F_s$ of these simplices. Let  $S$ be such a simplex,
situated in the tangent hyperplane of   $\Omega$ at the point  $\xi \in \Omega$
and containing  $\xi$. We shall compare its area $|S|$  with the
area  $|F_s(S)|$  of its image. 
With $v$ as before and $\varepsilon >0$, the convex hull of
 $S$ and the point $\xi + \varepsilon v$ is a  $k$-dimensional simplex  $S_\varepsilon$. Its volume is  $|S_\varepsilon| = \varepsilon |S| |v|$. Its image
$F_s(S_\varepsilon)$ is spanned by $F_s(S)$ and  
$F_s(\xi)+\varepsilon F_s(v)$, and
so has volume 
$|F_s(S_\varepsilon)| =  \varepsilon |F_s(S)| |F_s(v)|\cos\psi(s,\xi)$.

On the other hand, the quotient $|F_s(S_\varepsilon)|/|S_\varepsilon|$
equals the Jacobian of $F_s$, which is $\exp(\sum_1^k \lambda_\nu s)$.
Combining, one finds that
\begin{align*}
\frac{|F_s(S)|}{|S|} = 
\frac{\exp\left(\sum_1^k \lambda_\nu s\right)\,|v|}{|F_s(v)|\cos\psi(s,\xi)}
= &\,\exp\left(\sum_1^k \lambda_\nu s\right)\, \frac{\sqrt{\sum_1^k e^{-2 \lambda_j s}\lambda_j^2 x_j^2}}{\sqrt{\sum_1^k \lambda_j^2 x_j^2}}\\
= &\, \frac{\sqrt{\sum_{j=1}^k \exp\big[
{2(\sum_{\nu=1}^k \lambda_\nu -\lambda_j) s}\big]\, \lambda_j^2 x_j^2}}{\sqrt{\sum_1^k \lambda_j^2 x_j^2}}.
\end{align*}
It follows that 
\begin{equation*}
1 \le \frac{|F_s(S)|}{|S|} \le  e^{(k-1)\lambda_{\mathrm{max}}\, s}.
\end{equation*}
Summing over small simplices, we conclude that also
\begin{equation}\label{area}
1 \le \frac{| \Omega_s|}{|\Omega|} \le  e^{(k-1)\lambda_{\mathrm{max}}\, s},
\end{equation}
for any $s>0$.

Next, we estimate the factors in \eqref{muz}, 
still
assuming  $s>0$. First,  $ | F_s(v)| \le  e^{\lambda_{\mathrm{max}}\, s} |v|$ and
$ |v| \simeq |\xi| \simeq \sqrt \beta $, so that
\begin{equation*}
  | F_s(v)|  \lesssim  e^{\lambda_{\mathrm{max}}\, s} \sqrt \beta.
\end{equation*}
Further,
\begin{equation*}
R(e^{\lambda s}\xi)  = \sum_j \lambda_j   e^{2\lambda_j s}\, x_j^2
\ge  \sum_j
 \lambda_j (1+2\lambda_{\mathrm{min}}\,s)\, x_j^2
= (1+2\lambda_{\mathrm{min}}\,s)\, R(\xi) = (1+2\lambda_{\mathrm{min}}\,s)\beta,
\end{equation*}
since $R(\xi) = \beta$.

Inserted in \eqref{muz}, these two estimates 
 lead to
\begin{equation*}
\mu_R(Z)  \lesssim  
\sqrt{\beta} \,e^{-\beta} \int_0^\infty 
e^{\lambda_{\mathrm{max}}\, s - 2\lambda_{\mathrm{min}}\,\beta s}\, \int_{\Omega_s}
  \,dS_s\,ds. 
\end{equation*}
The inner integral here is $| \Omega_s|$, so we can use \eqref{area}
and observe that $| \Omega| \lesssim a^{k-1}$,  to get
\begin{equation*}
\mu_R(Z)  \lesssim  
\sqrt{\beta} \,e^{-\beta}\, a^{k-1} \int_0^\infty 
e^{(k \lambda_{\mathrm{max}} - 2\lambda_{\mathrm{min}}\,\beta) s}\,\,ds. 
\end{equation*}
We can assume that $\alpha$  is so large that {
$\lambda_{\mathrm{min}}\,\beta >k \lambda_{\mathrm{max}}$}, and then the last integral will
be less than  $1/(\lambda_{\mathrm{min}}\,\beta) \sim 1/\beta$, which proves the assertion.
\end{proof}

\bigskip

\section{The case of  large $t$.}\label{mixed. t large}
We prove part of 
 Theorem \ref{th:revisited},
 considering the supremum in \eqref{maxfcn} taken only 
over $t> 1$.
\begin{proposition}\label{propo-mixed-glob}
Let $k\in\{1,\ldots, n\}$. 
Then the maximal operator $$
\sup_{t> 1}
\big|
 \int_{ \R^n}
K_t^{k,\nu}
(x,u)\,  f(
u
)\, 
 \misgausskd
(u) 
 \,
\big|\,$$
 is of weak type $(1,1)$ with respect to the invariant measure
 $\misgaussk
$, uniformly in $\nu\in \Z^{n-k}$.
\end{proposition}
\begin{proof}
As before,
 $f$ is nonnegative,  supported in $\R^k\times\tilde {\mathcal C}_\nu$ and
normalized in ${L^1( \misgaussk)}$. 
We need only consider
points $x=(\xi,  x_{\text{loc}})\in 
 \mathcal E
 \times \mathcal C_\nu$ and $u=(\eta, u_{\text{loc}})\in \R^k\times
 \tilde{\mathcal C_\nu}$.
Moreover, we shall use  for both  $x$ and $u$ 
 the coordinates introduced in  \eqref{def-coord} with $\beta=\log\alpha$, that is,
\begin{equation*}
\xi=e^{\lambda s}\tilde \xi,
\qquad
\eta=e^{\lambda s'}
\tilde \eta, 
\end{equation*}
where $\tilde \xi, \tilde \eta
\in E_{\log \alpha}$ and  $s,s' \in\R$.
Observe here that $|s| < C$, since $\xi \in \mathcal E$.
Then \eqref{Mehler-Di}
 and  
 the fact that $t> 1$
imply
\begin{align*}
K_t^{k,\nu}
(x,u)
&\lesssim 
\exp 
(
R(\xi))
\exp
 \Big(
 -{\sum_{j=1}^{k}{\lambda_j}
(x_j-e^{-\lambda_j t}u_j)^2 }\Big).
\end{align*}
Since 
$\xi\in\mathcal E$
and $e^{-\lambda t}\eta=e^{\lambda (s'-t)}\tilde \eta$,
we can apply  
Lemma \ref{lemma-peter-coord}\,$(a)$ 
getting
\begin{align*}{\sum_{j=1}^{k}{\lambda_j}
(x_j-e^{-\lambda_j t}u_j)^2 }
\ge
\lambda_{\mathrm{min}}
\big|
{ \xi-
e^{-\lambda t}\eta}
\big|^2
\gtrsim
\big|
\tilde \xi-\tilde \eta\big|^2,
\end{align*}
so that
\begin{align*}
K_t^{k,\nu}
(x,u)
&\lesssim   \exp
(R(\xi))\exp
 \big(
- c\,\big|
\tilde \xi- \tilde \eta\big|^2
\big) .
\end{align*}
By integrating we obtain 
\begin{align*}
\int
 K_t^{k,\nu}
(x,u)
f(u)\,\misgausskd
 (u)
&\lesssim
\exp
 \big(
R (e^{\lambda s}\,\tilde \xi)
 \big)
\int 
\exp
 \big(
- c\, \big|
\tilde \xi- \tilde \eta\big|^2
\big)\,f(u)\,\misgausskd
 (u).
\end{align*}
The right-hand side here is increasing in $s$,
and therefore the inequality
\begin{equation}\label{equality-s-alfa}
\exp
 \big(
R (e^{\lambda s}\,\tilde \xi)
 \big)
\int
\exp
 \big(
-c\,
\big|
\tilde \xi- \tilde \eta\big|^2
\big)\,f(u)\,\misgausskd
 (u)
>
\alpha
\end{equation} holds
if and only if $s>
s_\alpha (\tilde \xi)$
for some 
$s_\alpha (\tilde \xi)$,
with equality for $s=s_\alpha (\tilde \xi)$.
Since $\alpha>1$ and the last integral is less than $\| f\|_{L^1 (\gamma_\infty^k)}=1$, it follows that 
$s_\alpha (\tilde \xi)>0$.

We see that the set
of $x$ where the supremum in the statement of Proposition
\ref{propo-mixed-glob}
is larger than
 $C\alpha$ for some $C$
 is contained in the set $\mathcal A^{k,\nu}(\alpha)$
of points $(\xi,x_{\mathrm{loc}}
)\in
 \mathcal E
\times
\mathcal C_\nu\!
$ satisfying \eqref{equality-s-alfa}.

Applying  \eqref{equiv:lebesgue}, where now
$\big|
e^{\lambda s}\tilde \xi
\big|\simeq 
\sqrt{\log\alpha}$
and $\beta=\log\alpha$,
and observing that
$\big|\tilde {\mathcal C_\nu}\big|\lesssim 1$, 
we conclude
that
\begin{align*}
\gamma_\infty^k (\mathcal A^{k,\nu}(\alpha))
&\lesssim { \sqrt{\log \alpha}}
\int_{E_{\log\alpha}}
{
\int_{s_\alpha (\tilde \xi)<s<C}
}
\exp
 \Big(
-\sum_{j=1}^{k}
{\lambda_j}
e^{2\lambda_j\,s}\tilde \xi_j^2
 \Big)\,ds\,
dS(\tilde \xi) .
\end{align*}
To estimate the integrand here, we observe that
 for $s_\alpha(\tilde \xi) < s <C$
the inequality
\begin{align*}
e^{2 \lambda_j\,s}=
e^{2  \lambda_j\,s_\alpha (\tilde \xi)}\,
e^{2 \lambda_j\,(s-s_\alpha  (\tilde \xi))}
&\ge
e^{2   \lambda_j\,s_\alpha (\tilde \xi)} \big(1+2 \lambda_j (s-s_\alpha  (\tilde \xi))\big)
\end{align*}
implies that
\begin{align*}
\exp&
 \Big(
 \!
-\sum_{j=1}^{k}
{\lambda_j}
e^{2\lambda_j\,s}\tilde \xi_j^2 
 \Big)
 \!
 \le
\exp \Big(-\sum_{j=1}^{k}
{\lambda_j}
e^{2\lambda_j\,s_\alpha  (\tilde \xi) } \tilde \xi_j^2\Big)
\exp\!\Big(\!-2\sum_{j=1}^{k}
{\lambda_j^2}
e^{2\lambda_j\,s_\alpha  (\tilde \xi) }(s-s_\alpha  (\tilde \xi)) \tilde \xi_j^2
\Big)\notag\\
& \le
\exp
 \big(-
R (e^{\lambda \,s_\alpha  (\tilde \xi) }\,\tilde \xi)
 \big)
\exp\Big(-2\lambda_{\mathrm{min}}
(s-s_\alpha  (\tilde \xi))
\,
R (e^{\lambda \,s_\alpha  (\tilde \xi) }\,\tilde \xi)
 \Big)
\notag\\
&\le
\exp
 \big(-
R (e^{\lambda \,s_\alpha  (\tilde \xi) }\,\tilde \xi)
 \big)
\,
\exp\big(-c
\,(s-s_\alpha  (\tilde \xi))\,  \log\alpha 
\big),\end{align*}
because $R (e^{\lambda \,s_\alpha  (\tilde \xi) }\,\tilde \xi)
\simeq \log\alpha$ 
here.
Thus 
\begin{align*}
\gamma_\infty^k (\mathcal A^{k,\nu}(\alpha))
&\lesssim
\sqrt{\log \alpha}
\int_{E_{\log\alpha}}
\int_{s>s_\alpha (\tilde \xi)}\!\!
\exp
 \big(-
R (e^{\lambda \,s_\alpha  (\tilde \xi) }\,\tilde \xi)
 \big)
\exp\big(-c
\,(s-s_\alpha  (\tilde \xi))\,  \log\alpha 
\big)\,ds\,
dS(\tilde \xi) \notag\\
&\lesssim \frac{1}{\sqrt{\log \alpha}}\int_{E_{\log\alpha}}
\exp
 \big(-
R (e^{\lambda \,s_\alpha  (\tilde \xi) }\,\tilde \xi)
 \big)
\,dS(\tilde \xi) .
\end{align*}
Next we combine this estimate
with the case of equality in  \eqref{equality-s-alfa}.
Changing then the order of integration, we finally get
\begin{align*}
\gamma_\infty^k (\mathcal A^{k,\nu}(\alpha))
&\lesssim
\frac{1}{\alpha \sqrt{\log \alpha}}
\int_{E_{\log\alpha}}
\int
\exp
 \big(
-c\, \big|\tilde \xi- \tilde \eta\big|^2
\big)\,f(u)\,\misgausskd
 (u)
dS(\tilde \xi)\\
&\lesssim
\frac{1}{\alpha \sqrt{\log \alpha}}
\int
 \int_{E_{\log\alpha}}
\!\exp
 \!\big(
-c\, \big|\tilde \xi- \tilde \eta\big|^2
\big) dS(\tilde \xi)\,f(u)\,\misgausskd
 (u)\\
&
\lesssim
\frac{1}{\alpha \sqrt{\log \alpha}}
\int
f(u)\,\misgausskd
 (u)\,,
\end{align*}
proving 
Proposition
\ref{propo-mixed-glob}.
\end{proof}

\bigskip

\section{{The case of small  $t$ }}\label{smallt}
The following proposition, combined with
Proposition
\ref{propo-mixed-glob},
will complete
the proof of
 Theorem \ref{th:revisited}.

\begin{proposition}\label{stima-tipo-debole-misto}
Let $k\in\{1,\ldots, n\}$. 
Then the maximal operator $$\sup_{t\le 1}
\big|
 \int_{ \R^n}
K_t^{k,\nu}
(x,u)\,f(
u
)\,
 \misgausskd
(u) 
 \,
\big|\,$$
 is of weak type $(1,1)$ with respect to the invariant measure
 $\misgaussk$, uniformly in $\nu\in \Z^{n-k}$.
\end{proposition}
\begin{proof}
We fix the multiindex $\nu\in\Z^{n-k}$.
As before,
 $f \in {L^1( \misgaussk)}$ is nonnegative,  supported in $\R^k\times\tilde {\mathcal C}_\nu$ and
normalized, and
we write 
$\eta=
(u_j)_{j=1}^k$
and
$e^{-\lambda t} \eta=
\big(e^{-\lambda_j t}u_j\big)_{j=1}^k$.
For $m_1,m_2\in \N$ and $0<t\le 1$, 
we introduce regions $\mathcal S^{m_1,m_2}_t$, depending also on $\nu$.
If $m_1,\, m_2>0$,  let
\begin{align*}
\mathcal S^{m_1,m_2}_t
&=\big\{
(x,u)\in  M_k
:     \:       
2^{m_1-1} \sqrt t< |\xi-e^{-\lambda t}\eta |\le 2^{m_1} \sqrt t\,,\,
\\
&
\qquad 
2^{m_2-1} \sqrt t<|x_{\text{loc}}-u_{\text{loc}}|\le 2^{m_2} \sqrt t\,,\;
x_{\text{loc}}
\in  \mathcal C_\nu
\, ,\,
u_{\text{loc}}\in \tilde { \mathcal C_\nu} 
\,\big\}.
\end{align*}
If $m_1=0$,
we replace the condition
$2^{m_1-1} \sqrt t< |\xi-e^{-\lambda t}\eta |\le 2^{m_1} \sqrt t$ by
$|\xi-e^{-\lambda t}\eta| \le  \sqrt t$.
Analogously,
if $m_2=0$, the inequalities
$2^{m_2-1} \sqrt t<|x_{\text{loc}}-u_{\text{loc}}|\le 2^{m_2} \sqrt t$ 
are replaced by
$|x_{\text{loc}}-u_{\text{loc}}|\le \sqrt t$.
Observe that for any fixed $t$ these sets form a partition of $(\R^k\times\mathcal C_\nu)\times (\R^k\times\tilde{\mathcal C}_\nu)
\cap M_k
$.

In the set ${\mathcal S^{m_1,m_2}_t}$ we can  apply   \eqref{Mehler-Di}, and also \eqref{local-bound} for the local coordinates, to get
\begin{align*}
K_t^{k,\nu} (x,u) &
\lesssim
\frac{
\exp
 (R(\xi))
}{  t^{n/2}
}
\exp 
\Big({
-c{2^{2 m_1} } -c
{2^{2 m_2} } }\Big).
\end{align*}
Thus for all $(x,u)\in\R^n\times \R^n$ and $t> 0$,
\begin{equation*}
K_t^{k,\nu} (x,u)\lesssim
\sum_{m_1,m_2}
{\mathcal K}_t^{{m_1,m_2}} 
(x,u)
\,,
\end{equation*}
where we define
\begin{align}\label{ktm1m2}
{\mathcal K}_t^{{m_1,m_2}} 
(x,u)
&=
\frac{
\exp  (R(\xi))
}{  t^{n/2}
}
\exp 
\Big({
-c{2^{2 m_1} } -c
{2^{2 m_2} } }\Big)
\,\chi_{\mathcal S^{m_1,m_2}_t}(x,u),
\end{align}
omitting the indices $\nu$ and $k$.

Therefore, we  need  only show that 
\begin{equation}\label{obiettivo-finale}
\gamma_\infty^k\left\{
x\in\R^n:
\sup_{t\le 1}
 \int\!
{\mathcal K}_t^{{m_1,m_2}} 
\!
(x,u) f(u)
\misgausskd
(u) 
\!
>\alpha \right\}
 \lesssim
\frac{1}{\alpha}\, \exp 
({
-c{2^{2 m_1} } \!-c
{2^{2 m_2} } })
,
\end{equation}
since this will allow summing in $m_1$, $m_2$ in the space  $L^{1,\infty}$.

Observe that 
${\mathcal K}_t^{{m_1,m_2}} 
(x,u)\neq 0$ implies
$(x,u)\in M_k$
and
$|\xi-e^{-\lambda t}\eta|\le 2^{m_1}\sqrt{t}$,
and then Lemma
\ref{lemma-sost-peter}
yields
\begin{align*}
1& \lesssim 
 (1+|\xi|)^4 t^2+(1+|\xi|)^2 2^{2m_1}t 
\le
 ((1+|\xi
|)^2 2^{2m_1}t)^2+
 (1+|\xi|)^2 2^{2m_1}t.
\end{align*}
From this
it follows that
\begin{equation}\label{stima-t-}
 (1+|\xi|)^2\, 2^{2m_1}\,t \gtrsim 1
\end{equation}
as soon as there exists a point $u$ with $\mathcal K_t^{m_1,m_2}(x,u)\neq 0$.
Then
$t\ge \varepsilon$ for some
$\varepsilon
>0$ which may depend on $m_1, \;m_2$ and $\alpha$.
We conclude that the supremum in  \eqref{obiettivo-finale}
can as well be taken over 
$\varepsilon\le t\le 1$,
and that this supremum is a continuous function
of $x\in  \mathcal E\times 
  \mathcal C_\nu$.
  
\medskip

To verify \eqref{obiettivo-finale},
our idea is
 to construct a finite sequence of pairwise disjoint sets
$\big(\mathcal B^{(\ell)}\big)_{\ell=1}^{\ell_0}$ in $\R^n$ and a sequence of sets $\big(\mathcal Z^{(\ell)}\big)_{\ell=1}^{\ell_0}$ in $\R^n$, called forbidden zones, which will 
contain the level set in \eqref{obiettivo-finale}.
We will  show that
  \begin{equation}\label{final-subset}
\left\{x=(\xi, x_{\mathrm {loc}})\in \mathcal E\times 
  \mathcal C_\nu:
\sup_{\varepsilon\le t\le 1}
\int {\mathcal K}_t^{{m_1,m_2}}   (x,u)\,f(u)\,\misgausskd (u)\,
\ge \alpha
\right\}\subset \bigcup_{\ell=1}^{\ell_0}\mathcal Z^{(\ell)},
\end{equation}
and that for each $\ell$
\begin{align}\label{stima-con-exp} 
&\misgaussk (\mathcal Z^{(\ell)})
\lesssim
\frac{1}{\alpha} 
\exp 
\Big({
-c{2^{2 m_1} } -c
{2^{2 m_2} } }\Big)
 \int_{\mathcal B^{(\ell)}}f(u) 
\,\misgausskd
 (u).
\end{align}
Since 
the $ \mathcal B^{(\ell)}$ will be pairwise disjoint, 
we 
will  then be able to
  conclude
\begin{align*} 
\misgaussk \Big(\bigcup_{\ell=1}^{\ell_0} \mathcal Z^{(\ell)}  \Big)
 &\lesssim
\frac{1}{\alpha} \exp 
\Big({
-c{2^{2 m_1} } -c
{2^{2 m_2} } }\Big)
 \sum_{\ell=1}^{\ell_0}
  \int_{\mathcal B^{(\ell)}}f(u)\,\misgausskd
 (u)\notag\\
 &\lesssim
\frac{1}{\alpha} \exp 
\Big({
-c{2^{2 m_1} } -c
{2^{2 m_2} } }\Big)
\|f\|_{L^1(\misgaussk)}
.\end{align*}
This
will
 imply \eqref{obiettivo-finale}
and finish the proof 
 of Proposition \ref{stima-tipo-debole-misto}.
\medskip

The  sets $\mathcal B^{(\ell)}$ and
 $\mathcal Z^{(\ell)}$
will be defined recursively, by means of 
points $x^{(\ell)}$,\; $\ell=1,\ldots, \ell_0$.
To find the first point $x^{(1)}$, we consider the minimum of the quadratic form
$R(\xi)$
in the compact set
  \begin{equation*}
\mathcal A_0 = 
 \{x\in \mathcal E\times 
  \mathcal C_\nu:
\sup_{\varepsilon\le t\le 1}
\int {\mathcal K}_t^{{m_1,m_2}}   (x,u)\,f(u)\,\misgausskd
\ge \alpha
\} .
\end{equation*}
(Should this set be empty,
\eqref{obiettivo-finale} is immediate.)

By continuity this minimum 
is attained at some point $x^{(1)}=\big(\xi^{(1)}\,,\,x^{(1)}_{\text{loc}}
\big)$
of $\mathcal A_0$.
Moreover,
 there is some $t$, called $t_{1}$, in  $[\varepsilon,1]$
for which the supremum is attained, so that
$$
\int \mathcal {K}_{t_1}^{{m_1,m_2}}  (x^{(1)},u)\,f(u)\,\misgausskd
 (u)\ge \alpha
.$$
Because of the expression \eqref{ktm1m2} for the kernel ${\mathcal K}_{t}^{{m_1,m_2}}$ and the definition of ${\mathcal S^{m_1,m_2}_t}$,
this implies
\begin{align}\label{mixed-bound-1}
\alpha
&\le
R(
\xi^{(1)}) \,
  {t_1^{-n/2}}
\exp 
\Big({
-c{2^{2 m_1} } -c
{2^{2 m_2} } }\Big)
\int_{\mathcal B^{(1)}}f(u)\,\misgausskd
 (u),
 \end{align}
where the set $\mathcal B^{(1)}$     is defined by
$$\mathcal B^{(1)}=
\left\{\big( \eta, \,u_{\text{loc}} \big)
\!\in \R^k\times 
 \tilde { \mathcal C_\nu} 
:|  \xi^{(1)}-e^{-\lambda t_1}\eta|\le 
2^{m_1 } \sqrt t_1\,,\;
|x^{(1)}_{\text{loc}}-u_{\text{loc}}|\le 2^{m_2} \sqrt t_1\,
\right\}.$$
Next we introduce
the first {\em forbidden zone} (the terminology is taken from \cite{Peter})
\begin{multline*}
\mathcal Z^{(1)}
=\Big\{   \big(e^{\lambda s}\eta, u_{\text{loc}}     
\big)\in \R^k\times 
 \tilde { \mathcal C_\nu} :\,s \ge 0,\; R{( \eta)}
= R(\xi^{(1)}),\; \\
| \eta- \xi^{(1)}|
 < A 2^{3m_1}    \sqrt{t_1}   \,,
\,|u_{\text{loc}}-x^{(1)}_{\text{loc}}|< B 2^{2m_1+m_2}  \sqrt{t_1}
  \Big\},
\end{multline*}
for some $A,B>0$ to be determined, depending only on the dimension and  the parameters of the semigroup. 

The construction now proceeds by recursion.
Assume that we have selected $x^{(h)}$, $\mathcal B^{(h)}$
and $\mathcal Z^{(h)}$
for $h=1,\ldots, \ell-1$.
The definition of the point $x^{(\ell)}$ is analogous to that of 
$x^{(1)}$ above, except that the forbidden zones
 $\mathcal Z^{(h)}$, $h=1,\ldots, \ell-1$, are now excluded.
More precisely, if  the set
  \begin{equation}\label{def:set}
\mathcal A_\ell = 
\left\{x\in ( \mathcal E\times 
  \mathcal C_\nu ) 
  \setminus   \bigcup_{h=1}^{\ell-1} \mathcal Z^{(h)}:
\sup_{\varepsilon\le t\le 1}
\int {\mathcal K}_t^{{m_1,m_2}}   (x,u)\,f(u)\,\misgausskd (u)
\ge \alpha
\right\}
\end{equation}
is nonempty, we choose
 $x^{(\ell)}=\big(
\xi^{(\ell)}, x_{\mathrm{loc}}^{(\ell)}
\big)
$ 
as a point
 minimizing  $R(\xi)$ in $\mathcal A_\ell$.
 But if  $\mathcal A_\ell = \emptyset$,
the process stops at $\ell_0=\ell-1$.
We shall soon see that this actually occurs for some finite 
  $\ell_0$, which will depend on  $m_1, \;m_2$ and $\alpha$.

Assume now that $\mathcal A_\ell \ne \emptyset.$  We verify below that 
 $\mathcal A_\ell$ is compact, so that   $x^{(\ell)}$ can be chosen.
Then
 there is some  $t_{\ell}\in [\varepsilon, 1]$ for which
$$
\int {\mathcal K}_{t_\ell}^{{m_1,m_2}} (x^{(\ell)},u)\,f(u)\,\misgausskd (u)\ge \alpha.
$$
We observe that \eqref{stima-t-}
applies to $t_\ell$ and
$x^{(\ell)}$, 
so that
\begin{equation}\label{stima-t-prop}
(1 +  |\xi^{(\ell)}|)^2\,
  2^{2m_1}\,t_\ell  \gtrsim 1.
\end{equation}
Further, we define 
\begin{align*}
\mathcal B^{(\ell)}=&
\{
\big(
\eta
,
u_{\text{loc}}\big)
\in \R^k\times \tilde { \mathcal C_\nu}:\,
|
  \xi^{(\ell)}
  -e^{-\lambda t_\ell}\eta
 |
 \le 2^{m_1} \sqrt{t_\ell}\,,\; 
 |x_{\text{loc}}^{(\ell)}-u_{\text{loc}}|\le 2^{m_2} \sqrt{t_\ell}
\}\,,
\end{align*}
and
the associated forbidden region is
\begin{align*}
\mathcal Z^{(\ell)}&=\{\big(e^{\lambda s}\eta, u_{\text{loc}}
\big)
\in \R^k\times \tilde { \mathcal C_\nu}
:\,
s \ge 0
,\; R{( \eta
)}= R(\xi^{(\ell)}),
\,| \eta- \xi^{(\ell)}
|< A 2^{3m_1} \sqrt{t_\ell},\;
\\
&\qquad
\,|u_{\text{loc}}-x^{(\ell)}_{\text{loc}}|<B  2^{2m_1+m_2}  \sqrt{t_\ell}
\}.
\end{align*}

To see that  $\mathcal A_\ell$ is closed and thus compact, observe that
for $1\le h \le \ell - 1$ the minimum property of $x^{(h)} $ implies that
$\mathcal A_\ell \subset \mathcal A_h  \subset\{x= (\xi,x_{\mathrm{loc}}):
\: R(\xi) \ge R(\xi^{(h)}) \}.$ Thus
\begin{multline*}
  \mathcal A_\ell = \mathcal A_\ell \cap \{x= (\xi,x_{\mathrm{loc}}):
\: R(\xi) \ge R(\xi^{(h)}), \;1 \le h \le \ell-1 \} = \\
\bigcap_{h=1}^{\ell-1} \left\{x \in ( \mathcal E\times  \mathcal C_\nu )\setminus
 \mathcal Z^{(h)}: \:  R(\xi) \ge R(\xi^{(h)}),\;\sup_{\varepsilon\le t\le 1}
\int {\mathcal K}_t^{{m_1,m_2}}   (x,u)\,f(u)\,\misgausskd (u)
\ge \alpha  \right\}.
\end{multline*}
The sets in this this intersection are all closed because of the definition of
$ \mathcal Z^{(h)}$, and so  $ \mathcal A_\ell$ is closed. This completes the
description of the recursive procedure.

In analogy with \eqref{mixed-bound-1}
we have
\begin{align}\label{mixed-bound-ell-1}
\alpha&\le
{
\exp \big({
R(\xi^{(\ell)})   }\big)
}{ \, t_\ell^{-n/2}
}
\exp 
\big({
-c{2^{2 m_1} } -c
{2^{2 m_2} } }\big)
\int_{\mathcal B^{(\ell)}
}f(u)\,\misgausskd
 (u).
\end{align}
We now verify that the sets $\mathcal B^{(\ell)}$ 
and
$\mathcal Z^{(\ell)}$ 
have the required 
properties.

\begin{lemma}\label{lemma:disjoint}
The collection of sets 
$\mathcal B^{(\ell)}$ 
is pairwise disjoint.
\end{lemma}
\begin{proof}

We prove that
 any two sets 
$\mathcal B^{(\ell)}$ 
and $\mathcal B^{(\ell')}$  with $\ell<\ell'$
 are disjoint.
Since
 \begin{equation*}
\big|
 \xi^{(\ell)}-
e^{-\lambda t_{\ell}} 
\eta\big|
=
\big|
e^{-\lambda t_{\ell}} \big(
e^{\lambda t_{\ell}} 
 \xi^{(\ell)}-
\eta\big)\big|
\ge
e^{- {\lambda_{\max}} t_{\ell}} 
\big|
e^{\lambda t_{\ell}} 
 \xi^{(\ell)}-
\eta\big| 
\end{equation*}
for $t\le 1$, the projection of
$\mathcal B^{(\ell)}$ 
in $\R^k$ is contained in
a ball 
with center 
$
e^{\lambda t_\ell} \xi^{(\ell)}$ and radius $ 2^{m_1 }e^{\lambda_{\max}}  \sqrt{t_\ell}$.
Moreover, the projection of
$\mathcal B^{(\ell)}$ 
in $\R^{n-k}$
is contained in a ball 
with center 
$
 x_{\mathrm{loc}}^{(\ell)}$ and radius $ 2^{m_2 } \sqrt{t_\ell}$.
The projections of $\mathcal B^{(\ell')}$ have analogous properties.

Thus it is enough to prove that  the centers of these balls in $\R^k$ and $\R^{n-k}$
 are far from each other; more precisely, that 
 \begin{equation}\label{stima-distanza-tang}
\big|
e^{\lambda t_\ell} \xi^{(\ell)}-
e^{\lambda t_{\ell'}} \xi^{(\ell')}
\big|
\ge 
 2^{m_1} {{ e^{\lambda_{\max}} }}(\sqrt {t_\ell}+
 \sqrt{ t_{\ell'}}),
\end{equation}
or
 \begin{equation}\label{stima-distanza-tang-loc}
\big|
 x_{\mathrm{loc}}^{(\ell)}- x_{\mathrm{loc}}^{(\ell')}
\big|
\ge 
 2^{m_2} (\sqrt {t_\ell}+
 \sqrt{ t_{\ell'}}).
\end{equation}
Using the coordinates from Subsection \ref{section:polar} with $\beta=R(\xi^{(\ell)})$,
we write
\begin{equation*}
\xi^{(\ell')}=e^{\lambda s} \tilde\xi^{(\ell')} 
\end{equation*}
for some $\tilde\xi^{(\ell')}$ with
$R(\tilde\xi^{(\ell')})=
R(\xi^{(\ell)})$ and some $s \in\R$.
Here  $s\ge 0$, because
 $R(\xi^{(\ell')})\ge R( \xi^{(\ell)})$.
Since $x^{(\ell')}$ is not in the forbidden zone $ \mathcal Z^{(\ell)}$,
we must have
\begin{equation}\label{hypo1}
|
\tilde\xi^{(\ell')}-
\xi^{(\ell)}
|\ge A
 2^{3m_1} \sqrt {t_\ell}
 \end{equation} 
 or
 \begin{equation}\label{hypo2}
  \,|x^{(\ell')}_{\text{loc}}-x^{(\ell)}_{\text{loc}}|\ge B  2^{2m_1+m_2}  \sqrt{t_\ell}.
  \end{equation}

 Assume first that ${t_{\ell'}}
\ge M 2^{4m_1}  t_\ell$, for some $M\ge 2$ to be chosen.
Together with Lemma \ref{lemma-peter-coord} $(b)$,  this assumption implies
\begin{align*}
\big|
e^{\lambda t_\ell} \xi^{(\ell)}-
e^{\lambda  t_{\ell'}} \xi^{(\ell')}
\big|&=
\big|
e^{\lambda t_\ell} \xi^{(\ell)}-
e^{\lambda ( t_{\ell'}+s)} \tilde\xi^{(\ell')}
\big|
\gtrsim  
|\xi^{(\ell)}|\,( t_{\ell'}+s-t_\ell)
\gtrsim  
|\xi^{(\ell)}|\,
{ t_{\ell'}}
.\end{align*}
We now apply the assumption again and then
\eqref{stima-t-prop}, observing that $|\xi^{(\ell)}| \simeq \log \alpha > 1$
because $\xi^{(\ell)} \in \mathcal E$.
This gives
\begin{align*}
\big|
e^{\lambda t_\ell} \xi^{(\ell)}-
e^{\lambda  t_{\ell'}} \xi^{(\ell')}
\big|&\gtrsim
 |\xi^{(\ell)}|\,\sqrt{M}
\, 2^{2m_1}
 \sqrt{ t_{\ell}}\, \sqrt{ t_{\ell'}}\\
&\gtrsim 
\sqrt{M}\,
 2^{m_1}\,
 \sqrt{ t_{\ell'}}
 \\
&\gtrsim
 \sqrt{M}\,
 2^{m_1}\,
( 
\sqrt{ t_{\ell'}} 
+\sqrt{ t_{\ell}}) 
.\end{align*}
Fixing $M$ conveniently, depending on  the implicit constants, 
we obtain  \eqref{stima-distanza-tang}.

In the remaining case  ${t_{\ell'}}
< M 2^{4m_1}  t_\ell$, we have
\begin{equation*}
\sqrt{t_{\ell}}
>
\frac{2^{-2m_1-1}}{\sqrt{M}}
( \sqrt {t_{\ell'}}+\sqrt {t_\ell}).
\end{equation*}
Applying this to  \eqref{hypo1}
or
\eqref{hypo2}, we arrive at
\eqref{stima-distanza-tang} or \eqref{stima-distanza-tang-loc}
by choosing $A=2
e^{\lambda_{\text{max}}}\sqrt{M}$
and $B=2\sqrt{M}$.
\end{proof}

We next verify that the sequence $(x^{(\ell)})$ 
is finite.
For $\ell<\ell'$,   
 we have
as in the preceding proof 
\eqref{hypo1}
or
\eqref{hypo2}.
In the  case of \eqref{hypo1}, Lemma \ref{lemma-peter-coord}\,$(a)$
implies
\begin{align*}
\big|
\xi^{(\ell')}-
 \xi^{(\ell)}
\big|&
\gtrsim A
 2^{3m_1} \sqrt {t_\ell}.
\end{align*}
Since $t_\ell\ge \varepsilon$, we see that in both cases the distance 
$\big|
x^{(\ell')}-
x^{(\ell)}
\big|$ is bounded below by a positive constant.
But all the $
x^{(\ell)}$
are contained in the bounded set 
$ \mathcal E\times 
  \mathcal C_\nu  $, so they are finite in number.
Thus the set considered in \eqref{def:set}
must be empty for some $\ell-1=\ell_0$. This implies 
\eqref{final-subset}.

We now prove 
\eqref{stima-con-exp} .
Observe that the global component of  the forbidden zone
$\mathcal Z^{(\ell)}$
corresponds to some  region $Z$,  as  defined in \eqref{zona}, where
 $a=A 2^{3m_1} \sqrt{t_\ell}$ and $\beta=R(\xi^{(\ell)})$. 
By applying
Lemma \ref{lemma-Peter-forbidden} and  taking also the local component  into account, we get
\begin{align*}
\misgaussk (  \mathcal Z^{(\ell)})&\lesssim
\frac{\big(  A 2^{3m_1} \sqrt{t_\ell}\big)^{k-1}}{\sqrt{ R(\xi^{(\ell)})}}\, \exp\left(
{ -{
R(\xi^{(\ell)})   }}\right)
\big(
 B 2^{2m_1+m_2}  \sqrt{t_\ell}\big)^{n-k}
\notag\\
&\lesssim  \frac{1}{\sqrt{\log\alpha}}
 \big(A 2^{3m_1}\big)^{k-1} \,
\big(B 2^{2m_1+m_2}\,\big)^{n-k} \,
t_\ell^{(n-1)/2}\, 
\exp\left(
{ -{
R(\xi^{(\ell)})   }}\right)
, \end{align*}
since $|\xi^{(\ell)}|\simeq \sqrt{\log \alpha}$.
Estimating the exponential here by  means of \eqref{mixed-bound-ell-1},
we obtain 
\begin{align*}\misgaussk& (\mathcal Z^{(\ell)})\!\!
&\lesssim \frac{1}{\alpha
{\sqrt{t_\ell\,\log\alpha}}} 
 (A 2^{3m_1})^{k-1} 
\big(B 2^{2m_1+m_2}\,\big)^{n-k} \,
e^{
{
-c{2^{2 m_1} } -c
{2^{2 m_2} } }}
\int_{\mathcal B^{(\ell)}
}\!f(u)\misgausskd
 (u).
\end{align*}
Applying  also
\eqref{stima-t-prop},
we finally conclude
\begin{align*} 
\misgaussk
 (\mathcal Z^{(\ell)})
&\lesssim
\frac{2^{m_1}}{\alpha} 
 \big(A 2^{3m_1}\big)^{k-1} \,
\big( B2^{
2m_1+m_2}\big)^{n-k} \,
 \exp 
({-c{2^{2 m_1} } -c {2^{2 m_2} }})
\int_{\mathcal B^{(\ell)}}f(u)\,\misgausskd
 (u)\,
\\
 &\lesssim
\frac{1}{\alpha} \exp 
({-c{2^{2 m_1} } -c {2^{2 m_2} }})
 \, \int_{\mathcal B^{(\ell)}}f(u)\,\misgausskd
 (u).\end{align*}
This proves 
\eqref{stima-con-exp} and ends  the proof of Proposition  \ref{stima-tipo-debole-misto}.
\end{proof}

Finally, combining Proposition \ref{propo-locale},
Proposition \ref{propo-mixed-glob},
and Proposition \ref{stima-tipo-debole-misto}, 
we complete  the proof of Theorem \ref{th:revisited}, and therefore also that of  Theorem \ref{weaktype1}.

\medskip

In the next section, we will need a variant of 
Theorem \ref{weaktype1}, where the Mehler kernel is slightly modified. 
The proof of  Theorem
\ref{weaktype1} also
yields the following result.

\begin{theorem}\label{mehler1111}
Let $\kappa>0$.
The maximal operator associated with the kernel
\begin{align}\label{Mehler1}
&\frac{
\exp \Big({\sum_{j=1}^{n}
{\lambda_j}x_j^2   }\Big)
}{\sqrt{\Pi_{j=1}^n(1-e^{-2 \lambda_j t})}
}
\exp
 \Big(-\kappa
 {\sum_{j=1}^{n}\frac{\lambda_j}{1-e^{-2 \lambda_j t}}
(x_j-e^{-\lambda_j t}u_j)^2 }\Big)
\,,\qquad \text{ $t>0$, } 
\end{align}
is of weak type
$(1,1)$ with respect to the  measure $\gamma_\infty$.
\end{theorem}

\bigskip

 \section{ The general case}\label{building}
 We go back to the setting of Section \ref{intro}
 and prove Theorem \ref{main-theorem}. Thus we assume that the semigroup 
 $\big(
\mathcal H_t^{Q,B}
\big)_{t> 0}$
is normal.

Metafune,  Pr\"uss,  Rhandi and  Schnaubelt
found in \cite{MPRS}
a decomposition of $\R^n$ into subspaces invariant under $
\mathcal H_t$
called building blocks. The restriction of 
$
\mathcal H_t$
to each building block has covariance $Q=I$ 
and drift $B=\lambda(R-I)$,
where  $\lambda>0$ and $R$ is a real skew-symmetric matrix. 
In  \cite{Mauceri-Noselli}
Mauceri and Noselli then decomposed
each building block into invariant subspaces
of dimensions $1$ and $2$,
in which the kernel of 
$\mathcal H_t$
has an explicit and rather simple form.

Combining the decompositions in \cite{MPRS}
and  \cite{Mauceri-Noselli},
the result is that
after a change of coordinates we will have covariance matrix $Q=I$
and a drift matrix of the form
\begin{equation*}
B = \text{diag}\big(
B_2, B_4,\ldots, B_{2m},  -\lambda_{2m+1}, \ldots,  -\lambda_{n}\big).
\end{equation*}
Here
$B_{2j}$, $j=1,\ldots, m$,
is a $2\times 2$ block
matrix of the form
\[B_{2j}=
\begin{pmatrix}
-\lambda_{2j}& q_j\\
-q_j&-\lambda_{2j}
\end{pmatrix}\,,
\]
with $\lambda_{2j}>0$ and $q_j\in\R\setminus\{0\}$.
Also $\lambda_i>0$ for $2m<i\le n$.

With $Q$ and $B$ of this form,  we will determine the kernel of
$\mathcal H_t$; as before the integration is with respect to $\gamma_\infty$.
To begin with, we consider the semigroup in  $\R^2$ with covariance matrix $I$
whose  drift matrix is $B_{2j}$. The corresponding invariant measure is
independent of $q_{j}$ and has density
 $\pi^{-1}
\lambda_{2j}\,
\exp \big(
{-
{\lambda_{2j}}
|x|^2   }\big)
$.
For this see  \cite[page 185]{Mauceri-Noselli}, where  our $\lambda_{2j}$
corresponds to $1/(2\alpha)$. As verified in  
\cite[(3.6) and (3.7)]{Mauceri-Noselli}, the kernel of this two-dimensional 
semigroup is 
\begin{align}\label{kt2j}&
K_t^{2j}(x,u)
=\frac{
\exp \Big({\lambda_{2j}
|x|^2   }
\Big)}{1-e^{-2\lambda_{2j} t}}
\exp
 \Big(-\frac{\lambda_{2j} }{1-e^{-2\lambda_{2j}  t}}
{
\big| x-e^{- \lambda_{2j} t}u\big|^2 }\Big)
\,\\
&\qquad\qquad  \times \exp\Big[-\lambda_{2j} \,
\frac{
e^{-\lambda_{2j} t} }{1-e^{-2\lambda_{2j} t}}
\Big(
(1-\cos ( q_{j} t )\big)
\langle x , u \rangle
+\sin ( q_{j}t) \,
x \wedge u
\Big)
\Big],\notag
 \end{align}
where $x,u\in\R^2$ and $x\wedge u
 =
 x_1 u_2 -x_2 u_1$.
In  \cite{Mauceri-Noselli}, $q_j=\theta$ and $\lambda_{2j}=1$; the
simple transformation needed to pass to any $\lambda_{2j}>0$ is indicated
in  \cite[page 185]{Mauceri-Noselli}. We shall use the following estimate
of $K_t^{2j}$; notice that the bound is independent of $q_j$.

\begin{proposition}\label{propo-kernel-isotropic1} 
For $x,u\in \R^2$  and $t>0$, one has
\begin{equation*}   
K_t^{2j}(x,u)
\le 
\frac{
\exp \big(
\lambda_{2j}|x|^2\big)   }{1-e^{-2\lambda_{2j} t}}\,
\exp\Big(-\frac12\, 
\frac{\lambda_{2j}}{1-e^{-2\lambda_{2j} t}}
\big|
x-
e^{-\lambda_{2j}t}u\big|^2
\Big) .\end{equation*}
\end{proposition}
\begin{proof}
Let $z = x-e^{-\lambda_{2j}t}u$,  so that $x$ can be replaced by  
$z + e^{-\lambda_{2j}t}u$.
 We then rewrite \eqref{kt2j} as
\begin{equation}\label{111}
K_t^{2j}(x,u)
= 
\frac{
\exp \big(
\lambda_{2j}|x|^2\big)   }{1-e^{-2\lambda_{2j} t}} \,
\exp\Big(-
\frac{\lambda_{2j}}{1-e^{-2\lambda_{2j} t}}\, F \Big),
\end{equation} 
with 
\begin{equation*}
  F = |z|^2 + e^{-\lambda_{2j} t}
\left[ (1-\cos ( q_{j} t ))(e^{-\lambda_{2j}t}\,|u|^2 + \langle z , u \rangle)
+\sin ( q_{j}t)\, z \wedge u \right].
\end{equation*}

 Let $\beta \in (-\pi, \pi]$ be the angle   
between the vectors $z$ and $u$, with the sign chosen so that 
$z\wedge u = |z| |u| \sin \beta$. Then
\begin{eqnarray*}
  F = |z|^2 + e^{-2\lambda_{2j} t} (1-\cos ( q_{j} t ))|u|^2 +
e^{-\lambda_{2j}t}|z| |u| \left[ (1-\cos ( q_{j} t )) \cos \beta
+\sin ( q_{j}t) \sin \beta  \right].
\end{eqnarray*}
But
\begin{multline*}
 (1 -\cos ( q_{j} t )) \cos \beta + \sin ( q_{j}t)\sin \beta
= \cos \beta -\cos( q_{j} t +  \beta) 
= 2 \sin( q_{j} t/2) \sin( \beta + q_{j} t/2).
\end{multline*}
Thus
\begin{eqnarray*}
  F \ge |z|^2 + e^{-2\lambda_{2j} t} (1-\cos ( q_{j} t ) )|u|^2 -
2e^{-\lambda_{2j}t}|z| |u|\, |\sin ( q_{j}t/2) |.
\end{eqnarray*}
Applying the inequality between the geometric and arithmetic means
to the last term here, we conclude
\begin{eqnarray*}
  F \ge |z|^2 + e^{-2\lambda_{2j} t} (1-\cos ( q_{j} t ))|u|^2 -
2e^{-2\lambda_{2j}t} |u|^2 \sin^2 ( q_{j}t/2) - |z|^2/2 =  |z|^2/2.
\end{eqnarray*}
Because of  \eqref{111}, this implies the proposition.
\end{proof}

Consider now the
 semigroup $\mathcal H_t$. The block diagonal structure of the drift 
matrix  $B$ implies that  $\mathcal H_t$ 
 is the product of commuting semigroups acting in  $\R^2$ and  $\R$.
Those  in    $\R^2$ are as just described, and those in  $\R$ are
 like the ones considered in Section \ref{particular}, with kernels given by 
\eqref{KtjQB}.
This implies a tensor product structure  both for the invariant measure
and for the kernel of  $\mathcal H_t$. Let   $\lambda_{2j-1}=\lambda_{2j}$
for $j=1,\ldots,m$. Then
the invariant measure of $\mathcal H_t$
will be given by the expression  \eqref{measure}. Further, Proposition 
\ref{propo-kernel-isotropic1} implies that the  kernel of   
$\mathcal H_t$ satisfies
 \begin{align*}
&
K_t(x,u) \le
\frac{
\exp \Big({  
\sum_{i=1}^n
\lambda_i
|x_i|^2   }
\Big)
}{
\sqrt{\Pi_{i=1}^n(1-e^{-2\lambda_i t})}
}\,
\exp
 \Big(-\frac12\,
 {
 \sum_{i=1}^n
 \frac{\lambda_i }{1-e^{-2\lambda_i  t}}
| x_i-
e^{- \lambda_i t}
u_i
|^2 }\Big),
 \end{align*}
for all $t>0$ and $x,u \in \R^n$.
Observing now that the last expression 
coincides with the kernel given by \eqref{Mehler1} with $\kappa=1/2$ ,
we conclude  the proof of Theorem \ref{main-theorem}
using Theorem \ref{mehler1111}.


\begin{thebibliography}{M\"uPeRi2}

\bibitem{Scotto}
H. Aimar, L. Forzani and R. Scotto,
On Riesz Transforms and Maximal Functions in the Context of Gaussian Harmonic
Analysis,
{\em Trans. Amer. Math. Soc. }
\textbf{359}, (2005)  2137--2154.

\bibitem{Lorenzi}  M. Bertoldi and L. Lorenzi,
{\em Analytical methods for Markov semigroups}
Pure and Applied Mathematics (Boca Raton), 283, Chapman \& Hall/CRC, Boca Raton, FL, 2007.


\bibitem{DaPrato}
G. Da Prato and A. Lunardi,
On the
 Ornstein-Uhlenbeck
operator in spaces of continuous functions,
{\em J. Funct. Anal.} \textbf{131}, (1995) 94--114.


\bibitem{DaPratoZ} 
G. Da Prato and J.  Zabczyk, 
{\em Stochastic Equations in Infinite Dimensions}, Cambridge University Press,
Cambridge (1992).


\bibitem{FGS}
 E. Fabes, C. Gutierrez and R. Scotto, Weak-type estimates for the Riesz 
transforms associated with the Gaussian measure. 
{\em Rev. Mat. Iber.} \textbf{10}, (1994) 229--281.

\bibitem{JLMS}
J. Garc\`ia-Cuerva, G. Mauceri, S. Meda, P. Sj\"ogren and J. L. Torrea,
Maximal Operators for the Holomorphic Ornstein--Uhlenbeck Semigroup,
{\em
J. London Math.}  \textbf{67}, (2003) 219--234.


\bibitem{Kolmo}
A. N. Kolmogorov,
Zuf\"allige  Bewegungen,
{\em Ann. of Math.} \textbf{116}, (1934) 116--117.





\bibitem{Lunardi}
A. Lunardi and V. Vespri,
Generation of strongly continuous semigroups by elliptic operators with
unbounded coefficients in $L^p (\R^n)$,
{\em Rend. Istit. Mat. Univ. Trieste} \textbf{28}, (1997)  251--279.






\bibitem{Mauceri-Noselli}
G. Mauceri and L. Noselli,
The maximal operator associated to a non symmetric Ornstein-Uhlenbeck
semigroup,
{\em J. Fourier Anal. Appl.} \textbf{15}, (2009) 179--200.

\bibitem{SoriaCR}
T. Men\`arguez, S. P\'erez, F. Soria, 
Pointwise and norm estimates for operators associated with the Ornstein--Uhlenbeck semigroup,
{\em C. R. Acad. Sci. Paris} \textbf{326}, S\'erie I, (1998) 25--30.

\bibitem{Soria}
T. Men\`arguez, S. P\'erez, F. Soria, 
The Mehler maximal function: a geometric proof of the weak
type $(1,1)$,
{\em J. Lond. Math. Soc.} 
\textbf{62},  (2000)
846--856.


   
\bibitem{Priola}
G. Metafune, D. Pallara, E. Priola,   
Spectrum of Ornstein--Uhlenbeck Operators in $L^p$ Spaces with  
Respect to Invariant Measures,
{\em J. Funct. Anal.} \textbf{196}, (2002)  40-60.


\bibitem{MPRS}
G. Metafune, J. Pr\"uss, A. Rhandi, and R. Schnaubelt,
The domain of the Ornstein-Uhlenbeck
operator on a $L^p$-space with invariant measure,
{\em Ann. Sc. Norm. Super. Pisa Cl. Sci.} \textbf{1}, (2002)   471--487.

\bibitem{Muckenhoupt}
B. Muckenhoupt,  Poisson integrals for Hermite and Laguerre expansions, {\em Trans.
Amer. Math. Soc.} \textbf{139}, (1969)  231--242.


\bibitem{OU}
L. S. Ornstein and G. E. 
Uhlenbeck, 
 On the theory of Brownian Motion,
{\em Phys. Rev.} \textbf{36},  (1930) 823--841.


\bibitem{Soria-Perez}
S. P\'erez and F. Soria, 
Operators associated with the Ornstein--Uhlenbeck semigroup,
{\em J. Lond. Math. Soc.} 
\textbf{61},  (2000)
857--871.


\bibitem{Peter} 
P. Sj\"ogren, 
On the maximal function for the Mehler kernel, in Harmonic Analysis,
Cortona 1982, (G. Mauceri and G. Weiss, eds.), Springer Lecture Notes in Mathematics 992,
(1983) 73--82.

\bibitem{Stein}
E. M. Stein,{\it
Topics in Harmonic Analysis Related to the Littlewood--Paley Theory},
Annals Math. Studies, 
Princeton Univ. Press, Princeton, (1970).

\bibitem{vanNeerven}
J. M. A. M. van Neerven
 and 
J.
Zabczyk, 
Norm discontinuity of Ornstein--Uhlenbeck semigroups,
{\em Semigroup Forum}
\textbf{59},
(1999) 389--403.

\end{thebibliography}
\end{document}